\documentclass[11pt]{amsart}
\usepackage{amsfonts}
\usepackage{amsmath}
\usepackage{amssymb}
\usepackage{amsthm}

\begin{document}
\bibliographystyle{halpha}

% usual blackboard bold things
\newcommand{\C}{\ensuremath{\mathbb{C}}}
\newcommand{\R}{\ensuremath{\mathbb{R}}}
\newcommand{\N}{\ensuremath{\mathbb{N}}}
\newcommand{\Q}{\ensuremath{\mathbb{Q}}}
\newcommand{\Z}{\ensuremath{\mathbb{Z}}}

% things for this topic
\newcommand{\g}{{\mathfrak{g}}}
\newcommand{\h}{{\mathfrak{h}}}
\newcommand{\goss}{(G_0)^{\mathrm{ss}}}
\newcommand{\go}{G_0}
\newcommand{\Gm}{{\mathbb{G}_m}}
\newcommand{\mysl}{{\mathfrak{sl}}}
\newcommand{\Fs}{F_{\mathrm{sep}}}
\renewcommand{\sc}{{\mathrm{sc}}}

% solely for the Borel--Tits reference
\font\gfont=rpncr at 12pt               % Roman font containing guillemets
\newcommand{\gml}{{\gfont\char171}}      % guillemetleft, rpncr font
\newcommand{\gmr}{{\gfont\char187}}      % guillemetright, rpncr font

% operator names
\newcommand{\GL}{\operatorname{GL}}
\newcommand{\SL}{\operatorname{SL}}
\newcommand{\Hom}{\operatorname{Hom}}
\newcommand{\charF}{\operatorname{char}}
\newcommand{\rk}{\operatorname{rk}}
\newcommand{\ad}{\operatorname{ad}}
\newcommand{\rootht}{\operatorname{ht}}
\newcommand{\Stab}{\operatorname{Stab}}
\newcommand{\Lie}{\operatorname{Lie}}
\newcommand{\im}{\operatorname{im}}

% miscellaneous
\newcommand{\gen}[1]{\langle #1\rangle}
\newcommand{\abs}[1]{\left|#1\right|}

% theorem environments, numbered with the equations
\newtheorem{thm}[equation]{Theorem}
\newtheorem{lem}[equation]{Lemma}
\newtheorem{cor}[equation]{Corollary}
\newtheorem{prop}[equation]{Proposition}

\title{Freudenthal triple systems by root system methods}
\author{Fred~W.~Helenius}
%\date{\today}
\begin{abstract}
For certain Lie algebras $\g$, we can use a grading
$\g=\g_{-2}\oplus\g_{-1}\oplus\g_0\oplus\g_1\oplus\g_2$ and define a quartic form and a
skew-symmetric bilinear form on $\g_1$, thereby constructing a Freudenthal triple system.
The structure of the Freudenthal triple system
is examined using root system methods available in the Lie algebra context.  In the
cases $\g=E_8$ (where $\g_1$ is the minuscule representation of $E_7$)
and $\g=D_4$, we determine the groups stabilizing the quartic form and
both the quartic and bilinear forms.
\end{abstract}
\maketitle{}

\section{Introduction}
Attempts to understand the $56$-dimensional minuscule representation
of $E_7$ have been based upon an axiomatization of the properties
of a bilinear form and a quartic form defined on it, resulting in
a so-called Freudenthal triple system.  While the minuscule
representation of $E_7$ is the prototype Freudenthal triple
system, another interesting example can be found in Bhargava's
work on higher composition laws (\cite{bhargava}); the $8$-dimensional
space with a quartic form defined at the outset also forms a
Freudenthal triple system, as observed by Markus Rost.

Freudenthal triple systems have been studied previously using such tools
as Jordan algebras (\cite{springer}, \cite{seligman}), tensor algebra (\cite{freud53}),
or by an axiomatic development (\cite{brown}, \cite{ferrar}), but
in this paper we exhibit Freudenthal triple systems that are subspaces
of Lie algebras with operations defined in terms of the Lie bracket;
this allows the Freudenthal triple system to be examined using
little more than root system computations.  Our approach applies
to the exceptional Lie algebras other than $G_2$ as well as to those of
types $B$ and $D$; in particular, we obtain both the $56$-dimensional
prototype and the Bhargava/Rost example.  As an application, in each of
these cases we determine the group stabilizing the quartic form and
the group stabilizing both forms.

We begin (Section~\ref{secprelim}) by using a $\Z/5\Z$-grading on the
Lie algebras in question to define a quartic form and a
bilinear form on the grade $1$ elements.  After establishing some
basic properties of these forms (Section~\ref{secforms}) and
characterizing the so-called strictly regular elements (Section~\ref{secsr}),
we are able to verify that we have a Freudenthal triple system (Section~\ref{secfts}).

We next show how to explicitly compute the quartic form
in the simply-laced case (Section~\ref{secq}).
An eigenspace decomposition (Section~\ref{seceigen}) into four subspaces
mirrors the construction of Freudenthal triple systems from Jordan algebras.
We characterize the orbits of the Freudenthal triple system under
a group action in Section~\ref{secorbits}.

The next three sections
examine the groups whose actions stabilize the Freudenthal triple system
operations (either exactly or up to scalar multiples).
In Section~\ref{secgroups}, we show that a linear transformation
that stabilizes the quartic form up to a scalar factor likewise
stabilizes the bilinear form.
In the case $\g=E_8$ (Section~\ref{secE8}), we find that $E_7$
is the group that stabilizes both the forms on the prototypical
Freudenthal triple system, the minuscule representation of $E_7$, which
was known; we also find the group which stabilizes just the quartic
form, which is new.
In Section~\ref{secD4}, we find the groups stabilizing one or both forms
in the case $\g=D_4$; these are new results.

All the results mentioned are proved under the assumption that
$\g$ is the Lie algebra of an algebraic group $G$ that is split
over a field $F$.  In the final section, we show that our results
apply equally well to non-split groups.

In \cite{ferrar}, Ferrar uses the axiomatic definition of a
Freudenthal triple system to study its structure.  Our approach
moves in the opposite direction:  we begin with a structure
defined within a Lie algebra, study its properties and eventually
show that it satisfies the Freudenthal triple system axioms.
Although our hypotheses are totally different, our choice of results
to prove was often guided by the content of Ferrar's article.  The
table below indicates results here that are parallel to those of Ferrar
as well as results in articles by Clerc (\cite{clerc}) and Krutelevich
(\cite{krutelevich}).

\begin{table}[ht!]\centering
\begin{tabular}{|c|c|}
\hline
Lemma \ref{lxxy}& \cite{ferrar}, Cor. 2.5\\
Prop. \ref{prank1}& \cite{ferrar}, Cor. 6.2\\
Lemma \ref{leq5}& \cite{ferrar}, (5)\\
Lemma \ref{lunique}& \cite{clerc}, Lemma 8.5(b); \cite{ferrar}, Lemma 3.6\\
Prop. \ref{peigen}& \cite{ferrar}, \S4\\
Prop. \ref{porbits}& \cite{clerc}, \S\S8,9; \cite{krutelevich}, Def. 22\\
Prop. \ref{pgroups}& \cite{ferrar}, Lemma 7.3\\
\hline
\end{tabular}
\vskip6pt
\caption{Parallel results in other papers}
\end{table}

\section{Preliminaries}\label{secprelim}

Here we establish notation and conventions that will be used throughout
and summarize the key results from other papers that are used.

Let $F$ be an arbitrary field of characteristic $\ne 2,3$, and
let $G$ be a simple, connected linear algebraic group that is split over $F$,
and let $\g$ be its Lie algebra.
Let $\Psi$ be the root system of $\g$ with respect to a fixed maximal
torus~$\h$; thus $\Psi\subset\h^\vee$.  Let $\rho$ be
the highest root with respect to a fixed base of~$\Psi$.
As is usual (see, for example, \cite{humlie}, \S9.1), we define
$\gen{\beta,\gamma}=2(\beta,\gamma)/(\gamma,\gamma)$ for
any nonzero $\beta,\gamma\in\h^\vee$.
We assume~$\g$ is not of type $A$ or $C$, so there is a unique simple
root~$\alpha$ such that $\gen{\alpha,\rho}=1$ and $\alpha$ is a long root.
We will also assume that the rank of $\g$ is at least $4$.  In the later
sections, we will assume that $\g$ is simply-laced and thus of type $D$ or $E$.

For each $\beta\in\Psi$, the {\it $\alpha$-height} of $\beta$ is given by
$\gen{\beta,\rho}$; in other words, $\alpha$-height is the coefficient of $\alpha$.
Thus the $\alpha$-height is one of $-2,-1,0,1,2$.
This gives a grading $\g = \g_{-2}\oplus\g_{-1}\oplus\g_0\oplus\g_1\oplus\g_2$,
where, for each $k\ne0$, $\g_k$ is the direct sum of the root subspaces for roots
of $\alpha$-height~$k$;  $\g_0$ is the direct sum of the root subspaces for roots
of $\alpha$-height~$0$ and of~$\h$.  Equivalently, each $\g_k$ contains all
$x\in\g$ for which $[h_\rho,x] = k x$.
Since $\gen{\beta,\rho}=-2$ (resp.~$2$) only when $\beta=-\rho$ (resp.~$\rho$), we
see that $\g_{-2}$ and $\g_2$ are one-dimensional, consisting of the
root subspaces corresponding to $-\rho$ and $\rho$, respectively.

We write $x_\beta$ for a representative of the root subspace corresponding
to $\beta\in\Psi$, and always assume that such representatives have been
chosen to lie in a Chevalley basis (see \cite{humlie}, \S25).

The grading on $\g$ allows us to define several operations on $\g_1$ in a
natural way.  First, we define a quartic form $q(x)$ for $x\in\g_1$ by
$(\ad x)^4(x_{-\rho}) = q(x)x_\rho$.  We also define a $4$-linear form
$q(x,y,z,w)$ by linearization.  To specify the
scalar factor, we set $q(x,x,x,x) = q(x)$ for all $x\in\g_1$.

We also define a skew-symmetric bilinear form $\gen{x,y}$ on $\g_1$ by
$[x,y] = \gen{x,y}x_\rho$.  This form turns out to be nondegenerate (Lemma~\ref{lnondeg});
thus we may also define a symmetric triple product $xyz$ on~$\g_1$ by
requiring $\gen{w,xyz}=q(w,x,y,z)$ for all $w,x,y,z\in\g_1$.
We will show (Theorem~\ref{tFTS}) that $\g_1$ equipped with these operations
is a Freudenthal triple system.  These operations depend upon
the choice of the Chevalley basis as follows:  if instead of $x_\rho$ 
we choose $c x_\rho$ as the basis element in the root subspace
corresponding to $\rho$, then the bilinear form is scaled by $c^{-1}$ and
the quartic form is scaled by $c^{-2}$.

By Theorem~$2$ of \cite{ABS}, if $F$ is algebraically closed then
the Levi complement of a parabolic subgroup of the
linear algebraic group $G$ acts on the unipotent radical of the parabolic
subgroup with finitely many orbits.
Let $G_0$ be the subgroup of $G$ that corresponds to $\g_0$, more precisely,
the centralizer of $h_\rho$ in $G$.  In terms of the Lie
algebra, $G_0$ acts on $\g_1$ and actually partitions $\g_1$ into finitely
many orbits.

In Theorem 2.6 of \cite{rohrle}, R\"ohrle gives the number of $G_0$-orbits
in $\g_1$ for each Lie algebra $\g$ satisfying our common hypotheses.
For the Lie algebras $E_6$, $E_7$, $E_8$, there are five orbits.
Each orbit is represented by an element of the form $\sum_{i=1}^k x_{\beta_i}$
for $k = 0,\ldots,4$ where $\{\beta_1,\beta_2,\beta_3,\beta_4\}$ is a set of
mutually orthogonal roots of $\alpha$-height~$1$ (\cite{rohrle}, Theorem 4.8).
We refer to these as orbit~$0$ through orbit~$4$.
We may, and frequently do, take $\beta_1=\alpha$.

For Lie algebras of type $D_n$, each orbit has a representative as above, but there
are either two ($n > 4$) or three ($n=4$) distinct orbits generated by sums with two
terms; that is, ``orbit $2$'' is split into two or three orbits in this case; we refer
to each of them as a level~$2$ orbit.  Similarly, for Lie algebras of
type $B_n$, $n\ge4$, or $F_4$ there are two level~$2$ orbits.

For all of the types, orbit $4$ is also represented by $x_\alpha + x_{\rho-\alpha}$
(\cite{rohrle}, Corollary 4.4).

The semisimple part of $G_0$, which we denote by $\goss$, also acts on $\g_1$;
here there are finitely many orbits in the projective space $\mathbb{P}(\g_1)$.
These projective orbits correspond to the nonzero orbits under the action of $G_0$.  The action of $\goss$ is of interest because of the following fact.

\begin{lem}\label{lgoss}
The quartic form and skew-symmetric bilinear form on $\g_1$
are preserved by the action of $\goss$.
\end{lem}
\begin{proof}
The elements of $\goss$ act on $\g$ by Lie algebra homomorphisms.
For any basis element of the Lie subalgebra of $\g$ corresponding to $\goss$,
i.e., any $x_\beta$ where $\beta$ is a root of $\alpha$-height~$0$ or any
$h_\gamma$ where $\gamma$ is a simple root other than $\alpha$, we have
$[x_\beta,x_\rho]=0$ and $[h_\gamma,x_\rho]=0$ because $\rho$ is orthogonal
to every root of $\alpha$-height~$0$.  Similarly, we also have
$[x_\beta,x_{-\rho}]=0$ and $[h_\gamma,x_{-\rho}]=0$.  Thus elements of
$\goss$ fix $x_\rho$ and $x_{-\rho}$.
The quartic form and bilinear form we have
defined on $\g_1$ depend only on the Lie bracket, $x_\rho$ and $x_{-\rho}$, so
both are preserved by the action of $\goss$.
\end{proof}

By Th\'eor\`eme 3.13 in Borel \& Tits (\cite{boreltits}),
the closure of any of the $G_0$-orbits is its union with all smaller (i.e., lower level) orbits.
In particular, the largest orbit, orbit 4, is dense in $\g_1$.

The statements about orbits above assume that $F$ is
algebraically closed.  For a general $F$, geometric statements
about orbits will at least be true over the algebraic closure of $F$.
The algebraic consequences, such as Lemma~\ref{lgoss} above, remain true
for any $F$, since they involve polynomial relations defined over $F$.
To avoid repetition, we make this convention:  \emph{all statements
about orbits are understood to refer to the orbits over the algebraic
closure.}

As mentioned earlier, we have assumed for convenience that $G$ is
split over $F$.  Our results apply more generally, as explained in
Section~\ref{secnonsplit}.

\section{The bilinear and quartic forms}\label{secforms}

Given any $x,y\in\g_1$, the Lie algebra product lies in
$\g_2=F x_\rho$; thus we may define a bilinear form $\gen{x,y}$
on $\g_1$ by $[x,y] = \gen{x,y}x_\rho$.  This form is clearly
skew-symmetric.

\begin{lem}\label{lnondeg}
The bilinear form $\gen{-,-}$ on $\g_1$ is nondegenerate.
\end{lem}
\begin{proof}
The elements $x_\beta$ with $\beta$ a root of $\alpha$-height~$1$ form a
basis for $\g_1$.  Consider the matrix of the form with respect to this basis;
the entries are of the form $\gen{x_\beta,x_\gamma}$ with $\beta,\gamma$ roots of
$\alpha$-height~$1$.  Such an entry is zero unless $[x_\beta,x_\gamma]$ is
a nonzero element of $F x_\rho$; that is, unless $\beta+\gamma=\rho$.
Since $\gen{\beta,\rho}=1$, $\rho-\beta$ is a root (of $\alpha$-height~$1$);
hence each row and each column of the matrix contains exactly one nonzero entry.
Such a matrix (sometimes called a monomial matrix) is the product of a
diagonal matrix with nonzero entries on the diagonal and a permutation matrix,
hence it is invertible.  Thus the form is nondegenerate.
\end{proof}

Since $x_{-\rho}$ is in $\g_{-2}$, for any $x\in\g_1$ the
value $[x,[x,[x,[x,x_{-\rho}]]]]$, or, more briefly,
$(\ad x)^4(x_{-\rho})$, is in $\g_2$.  Thus
we may define a quartic form $q(x)$ for $x\in\g_1$ by
$(\ad x)^4(x_{-\rho}) = q(x)x_\rho$.  This in turn gives rise to a
fully symmetric $4$-linear form $q(x,y,z,w)$ defined by setting
$q(x,x,x,x)=q(x)$ and extending by linearization.

\begin{lem}\label{lqsum}
Let $\beta_1,\beta_2,\beta_3,\beta_4$ be roots of $\alpha$-height~$1$.  The
value of the $4$-linear form $q(x_{\beta_1},x_{\beta_2},x_{\beta_3},x_{\beta_4})$
is given by
\[
q(x_{\beta_1},x_{\beta_2},x_{\beta_3},x_{\beta_4})x_\rho =
\frac1{4!}\sum_{\pi\in S_4} (\ad x_{\beta_{\pi(1)}}\circ\ad x_{\beta_{\pi(2)}}
\circ\ad x_{\beta_{\pi(3)}}\circ\ad x_{\beta_{\pi(4)}})(x_{-\rho}),
\]
where $S_4$ is the symmetric group on $\{1,2,3,4\}$.
\end{lem}
\begin{proof}
Let $\lambda,\mu,\nu$ be indeterminates.  By the definition of the quartic form,
we have
\[
q(x_{\beta_1}+\lambda x_{\beta_2}+\mu x_{\beta_3}+\nu x_{\beta_4})x_\rho =
(\ad x_{\beta_1}+\lambda x_{\beta_2}+\mu x_{\beta_3}+\nu x_{\beta_4})^4 (x_{-\rho}).
\]
Replacing the quartic form by the equivalent $4$-linear form,
the coefficient of $\lambda\mu\nu$ on the left-hand side is
$24 q(x_{\beta_1},x_{\beta_2},x_{\beta_3},x_{\beta_4})x_\rho$.
On the right-hand side, the coefficient of $\lambda\mu\nu$ is
$\sum_{\pi\in S_4} (\ad x_{\beta_{\pi(1)}}\circ\ad x_{\beta_{\pi(2)}}
\circ\ad x_{\beta_{\pi(3)}}\circ\ad x_{\beta_{\pi(4)}})(x_{-\rho})$.
Equating the coefficients yields the result.
\end{proof}

\begin{cor}\label{ctworho}
Let $\beta_1,\beta_2,\beta_3,\beta_4$ be roots of $\alpha$-height~$1$;
then the $4$-linear form $q(x_{\beta_1},x_{\beta_2},x_{\beta_3},x_{\beta_4}) = 0$ whenever
$\beta_1+\beta_2+\beta_3+\beta_4 \ne 2\rho$.
\end{cor}
\begin{proof}
If the summand $(\ad x_{\beta_{\pi(1)}}\circ\ad x_{\beta_{\pi(2)}}\circ
\ad x_{\beta_{\pi(3)}}\circ\ad x_{\beta_{\pi(4)}})(x_{-\rho})$ in the previous lemma
is nonzero, it must be some multiple of $x_\rho$; that is, we must have
$\beta_1+\beta_2+\beta_3+\beta_4+(-\rho)=\rho$.
\end{proof}

To establish that the quartic form is nonzero, we will require some information about
the structure constants that define the multiplication in~$\g$.
Given roots $\beta,\gamma\in\Psi$, we
denote the corresponding structure constant
by $c_{\beta,\gamma}$; that is, we define $c_{\beta,\gamma}$ so that
$[x_\beta,x_\gamma] = c_{\beta,\gamma}x_{\beta+\gamma}$.  In particular,
$c_{\beta,\gamma}=0$ if $\beta+\gamma$ is not a root.
As always, we are assuming that the elements $x_\beta$, $x_\gamma$,
etc. are in a Chevalley basis.
Theorem 4.1.2 in \cite{carter} provides the following useful
facts about these structure constants:
\begin{enumerate}
	\item[1.] If $\beta,\gamma\in\Psi$, then $c_{\beta,\gamma} = - c_{\gamma,\beta}$.
	\item[2.] If $\beta,\gamma,\delta\in\Psi$ are long roots such that $\beta+\gamma+\delta=0$, then
   $c_{\beta,\gamma} = c_{\gamma,\delta} = c_{\delta,\beta}$.
	\item[3.] If $\beta,\gamma\in\Psi$ are long roots and $\beta+\gamma$ is a root,
	 then $c_{\beta,\gamma}=\pm 1$.
	\item[4.] If $\beta,\gamma,\delta,\epsilon\in\Psi$ are long roots such that
	 $\beta+\gamma+\delta+\epsilon=0$ and no two are opposite, then
	 \begin{equation}\label{eqcarter4}
	 c_{\beta,\gamma} c_{\delta,\epsilon} +
	 c_{\gamma,\delta} c_{\beta,\epsilon} +
	 c_{\delta,\beta} c_{\gamma,\epsilon} = 0.
	 \end{equation}
\end{enumerate}
\noindent For $2$, $3$ and $4$ we have simplified the statements in \cite{carter} by
requiring the roots to be long.  Since $x_\beta,x_\gamma$ are in
a Chevalley basis, we also have $c_{\beta,\gamma} = - c_{-\beta,-\gamma}$
for all $\beta,\gamma\in\Psi$ and $[x_\beta,x_{-\beta}]=h_\beta$
(\cite{humlie}, \S25.2).  These facts will be used freely, usually without further comment.

We now use the facts above to compute the value of the $4$-linear form
on some special arguments.

\begin{lem}\label{lcase1}
If $\beta$ is a long root of $\alpha$-height~$1$, then
\begin{equation}\label{eqcomp1}
q(x_\beta,x_\beta,x_{\rho-\beta},x_{\rho-\beta}) = 1.
\end{equation}
\end{lem}
\begin{proof}
By hypothesis, $\gen{\beta,\rho}=1$, so $\rho-\beta$ is also a root.
We begin by finding $q(x_\beta+\lambda x_{\rho-\beta})$, which is given by
$(\ad x_\beta+\lambda x_{\rho-\beta})^4(x_{-\rho}) = q(x_\beta+\lambda x_{\rho-\beta})x_\rho$.
The left-hand side can be calculated
by repeatedly applying $\ad x_\beta+\lambda x_{\rho-\beta}$.  For the first step,
\[
[x_\beta+\lambda x_{\rho-\beta},x_{-\rho}] =
  c_{\beta,-\rho} x_{\beta-\rho} + \lambda c_{\rho-\beta,-\rho} x_{-\beta}.
\]
Writing $a$ for $c_{\beta,-\rho}$ and $b$ for $c_{\rho-\beta,-\rho}$,
we continue:
\begin{align*}
&[x_\beta+\lambda x_{\rho-\beta},a x_{\beta-\rho} + \lambda b x_{-\beta}] =
  \lambda a h_{\rho-\beta} + \lambda b h_\beta,\\
&[x_\beta+\lambda x_{\rho-\beta},\lambda a h_{\rho-\beta} + \lambda b h_\beta] =
  -2\lambda^2 a x_{\rho-\beta} -2\lambda b x_\beta
+\lambda a x_\beta + \lambda^2 b x_{\rho-\beta},\\
\begin{split}
[x_\beta+\lambda x_{\rho-\beta},-2\lambda^2 a x_{\rho-\beta}
  -2\lambda b x_\beta+\lambda a x_\beta + \lambda^2 b x_{\rho-\beta}]
=\qquad\qquad\qquad\quad\\
3\lambda^2 c_{\beta,\rho-\beta}(b-a)x_\rho.
\end{split}
\end{align*}
Since $\beta$, $-\rho$ and $\rho-\beta$ are long roots that sum to zero, we have
$a = c_{\beta,-\rho} = c_{\rho-\beta,\beta} = -c_{\beta,\rho-\beta}$ and
$b = c_{\rho-\beta,-\rho} = c_{\beta,\rho-\beta} = -a$.  Since the
structure constant $c_{\beta,\rho-\beta}$ is $\pm1$,
the result is $6 \lambda^2 c_{\beta,\rho-\beta}^2=6 \lambda^2$.
The term in $\lambda^2$ resulting from expanding
$q(x_\beta+\lambda x_{\rho-\beta})$ is $6 \lambda^2 q(x_\beta,x_\beta,x_{\rho-\beta},x_{\rho-\beta})$,
so we have
$q(x_\beta,x_\beta,x_{\rho-\beta},x_{\rho-\beta}) = 1$,
as required.
\end{proof}

Since there is always a long root of $\alpha$-height~$1$ (e.g., $\alpha$ itself),
we have established that the $4$-linear form and thus also the quartic form
are not identically zero.  In particular, taking $\beta = \alpha$ and $\lambda = 1$,
we have $q(x_\alpha + x_{\rho-\alpha}) = 6$.

In the next section we will also need to know that the $4$-linear form is
nonzero in another special case.  We show this after the next two lemmas.
The first is an easy but useful observation; the second is a fact about
structure constants that will also be used in Section~\ref{secq}.

\begin{lem}\label{lrhominus}
If $\beta$ is root of $\alpha$-height~$1$, then $\rho-\beta$ is also
a root, is also of $\alpha$-height~$1$, and has the same length as $\beta$.
If $\beta$ and $\gamma$ are orthogonal roots of $\alpha$-height~$1$,
then $\rho-\beta$ and $\rho-\gamma$ are also orthogonal.
\end{lem}
\begin{proof}
We have $\gen{\beta,\rho}=1$, so $\rho-\beta$ is a root.
The $\alpha$-height of $\rho-\beta$ is $\gen{\rho-\beta,\rho}=
\gen{\rho,\rho}-\gen{\beta,\rho}=2-1=1$.
The highest root $\rho$ is long, so if $\beta$ is long, then so is
$\rho-\beta$.  If $\beta$ is short, $\rho-\beta$ cannot
be long, for we then have that $\rho-(\rho-\beta)=\beta$ is long.

If $\gen{\beta,\gamma}=0$, then
\begin{align*}
\gen{\rho-\beta,\rho-\gamma}&=\frac2{(\rho-\gamma,\rho-\gamma)}(\rho-\beta,\rho-\gamma)\\
&=\frac2{(\rho-\gamma,\rho-\gamma)}((\rho,\rho)-(\gamma,\rho)-(\beta,\rho)+(\beta,\gamma))\\
&=\frac{(\rho,\rho)}{(\rho-\gamma,\rho-\gamma)}(\gen{\rho,\rho}-\gen{\gamma,\rho}-\gen{\beta,\rho})\\
&= 0.\qedhere
\end{align*}
\end{proof}

\begin{lem}\label{laux2}
Let $\beta$ and $\gamma$ be two orthogonal long roots of $\alpha$-height~$1$.
Each of $\beta-\rho$, $\gamma-\rho$ and $\rho-\beta-\gamma$ is a root;
each of the structure constants $c_{\beta,\gamma-\rho}$,
$c_{\gamma,\beta-\rho}$, $c_{\beta,-\rho}$, $c_{\gamma,-\rho}$ 
is $\pm 1$ and their product is $1$.
\end{lem}
\begin{proof}
By Lemma~\ref{lrhominus}, $\rho-\beta$ and $\rho-\gamma$ are long roots,
so their negatives are as well.
Since $\beta$ and $\gamma$ are orthogonal, $\gen{\rho-\beta,\gamma}=
\gen{\rho,\gamma}=1$, so $\rho-\beta-\gamma$ is a root.
Since these are roots, the specified structure constants are nonzero;
since all roots involved are long, they are $\pm 1$.

We apply $\eqref{eqcarter4}$, replacing $\beta,\gamma,\delta,\epsilon$ with
$\rho-\beta-\gamma,\beta,\gamma,-\rho$ to yield
\[
	 c_{\rho-\beta-\gamma,\beta} c_{\gamma,-\rho} +
	 c_{\beta,\gamma} c_{\rho-\beta-\gamma,-\rho} +
	 c_{\gamma,\rho-\beta-\gamma} c_{\beta,-\rho} = 0.
\]
The structure constants in the middle term are zero since $\beta+\gamma$ is not a root.
Thus we find
$c_{\rho-\beta-\gamma,\beta} c_{\gamma,-\rho} =
-c_{\gamma,\rho-\beta-\gamma} c_{\beta,-\rho}$.
Substituting $c_{\rho-\beta-\gamma,\beta}=c_{\beta,\gamma-\rho}$ and
$c_{\gamma,\rho-\beta-\gamma} = c_{\beta-\rho,\gamma} = -c_{\gamma,\beta-\rho}$,
we have
$c_{\beta,\gamma-\rho} c_{\gamma,-\rho} = c_{\gamma,\beta-\rho} c_{\beta,-\rho}$.
Since each side is $\pm1$, the product of all four structure constants is~$1$.
\end{proof}

\begin{lem}\label{lcase2}
If $\beta$ and $\gamma$ are two orthogonal long roots of $\alpha$-height~$1$,
then
\begin{equation}\label{eqcomp2}
q(x_\beta,x_\gamma,x_{\rho-\beta},x_{\rho-\gamma}) =
  -\frac12 c_{\beta,-\rho} c_{\gamma,-\rho} \ne 0.
\end{equation}
\end{lem}
\begin{proof}
By Lemma~\ref{lqsum},
there are $24$ terms to consider.  We divide them into three classes.

Class $1$:  These are the terms in which the first two elements
applied to $x_{-\rho}$ are $x_\beta$ and $x_{\rho-\beta}$, in either order, or, likewise,
$x_\gamma$ and $x_{\rho-\gamma}$.  The result in $\g_0$ is thus in $\h$.
By Lemma~\ref{lrhominus}, since $\beta$ and $\gamma$ are orthogonal, so are $\rho-\beta$ and $\rho-\gamma$.
As a result, half the terms in this case are zero; e.g.,
$[x_{\rho-\beta},[x_\beta,x_{-\rho}]]$ is a multiple of $h_{\rho-\beta}$,
and $[x_{\rho-\gamma},h_{\rho-\beta}]=0$.  The other terms are
\begin{align*}
[x_{\rho-\gamma},[x_\gamma,[x_{\rho-\beta},[x_\beta,x_{-\rho}]]]] &=
  - c_{\rho-\gamma,\gamma} c_{\beta,-\rho} x_\rho,\\
[x_\gamma,[x_{\rho-\gamma},[x_\beta,[x_{\rho-\beta},x_{-\rho}]]]] &=
  - c_{\gamma,\rho-\gamma} c_{\rho-\beta,-\rho} x_\rho,\\
[x_{\rho-\beta},[x_\beta,[x_{\rho-\gamma},[x_\gamma,x_{-\rho}]]]] &=
  - c_{\rho-\beta,\beta} c_{\gamma,-\rho} x_\rho,\\
[x_\beta,[x_{\rho-\beta},[x_\gamma,[x_{\rho-\gamma},x_{-\rho}]]]] &=
  - c_{\beta,\rho-\beta} c_{\rho-\gamma,-\rho} x_\rho.
\end{align*}

We have
$c_{\gamma,-\rho} = c_{\rho-\gamma,\gamma} = 
- c_{\gamma,\rho-\gamma} = - c_{\rho-\gamma,-\rho}$,
and likewise with $\gamma$ replaced by $\beta$.  Thus each of these four terms
is equal to $-c_{\gamma,-\rho} c_{\beta,-\rho} x_\rho$.

Class 2:  Here the terms are those in which the first two elements
applied to $x_{-\rho}$ are $x_\beta$ and $x_{\rho-\gamma}$, in either order, or, likewise,
$x_\gamma$ and $x_{\rho-\beta}$.  Since $\beta-\gamma$ (resp.~$\gamma-\beta$) is not
a root, these terms are all zero.

Class 3:  The remaining terms are those in which the first two elements
applied to $x_{-\rho}$ are $x_\beta$ and $x_{\gamma}$, in either order, or, likewise,
$x_{\rho-\gamma}$ and $x_{\rho-\beta}$.  Since $\beta + \gamma - \rho$ is a root
by Lemma~\ref{laux2}, the
result in $\g_0$ is nonzero and not in~$\h$, so we compute each term by
simply accumulating the structure constants.  Four of the terms are
\begin{align*}
[x_{\rho-\gamma},[x_{\rho-\beta},[x_\gamma,[x_\beta,x_{-\rho}]]]] &=
  c_{\rho-\gamma,\gamma} c_{\rho-\beta,\beta+\gamma-\rho} c_{\gamma,\beta-\rho} c_{\beta,-\rho} x_\rho\\
  &= c_{\gamma,-\rho} c_{-\gamma,\rho-\beta} c_{\gamma,\beta-\rho} c_{\beta,-\rho} x_\rho\\
  &= -c_{\gamma,-\rho} c_{\beta,-\rho} x_\rho,\\
[x_\gamma,[x_\beta,[x_{\rho-\gamma},[x_{\rho-\beta},x_{-\rho}]]]] &=
  -c_{\gamma,-\rho} c_{\beta,-\rho} x_\rho,\\
[x_{\rho-\beta},[x_{\rho-\gamma},[x_\gamma,[x_\beta,x_{-\rho}]]]] &=
  -c_{\beta,\gamma-\rho} c_{\gamma,\beta-\rho},\\
[x_\beta,[x_\gamma,[x_{\rho-\gamma},[x_{\rho-\beta},x_{-\rho}]]]] &=
  -c_{\beta,\gamma-\rho} c_{\gamma,\beta-\rho};
\end{align*}
the remaining four are obtained by interchanging $\beta$ and $\gamma$.  This yields
four terms equal to $- c_{\beta,-\rho} c_{\gamma,-\rho} x_\rho$ and four equal to
$-c_{\beta,\gamma-\rho} c_{\gamma,\beta-\rho} x_\rho$.

Combining all the terms, we have
\[
q(x_\beta,x_\gamma,x_{\rho-\beta},x_{\rho-\gamma}) =
  -\frac13 c_{\beta,-\rho} c_{\gamma,-\rho}
   - \frac16 c_{\beta,\gamma-\rho} c_{\gamma,\beta-\rho};
\]
but it follows from Lemma~\ref{laux2} that the two products of
structure constants are equal.  Thus we have
\[
q(x_\beta,x_\gamma,x_{\rho-\beta},x_{\rho-\gamma}) = -\frac12 c_{\beta,-\rho} c_{\gamma,-\rho}.
\]
In particular, it is not zero.
\end{proof}

\section{Strictly regular elements}\label{secsr}

For any fixed $x,y,z\in\g_1$, the expression $q(w,x,y,z)$ with $w\in\g_1$
is a linear function of $w$.
Since the skew-symmetric bilinear form $\gen{-,-}$ is nondegenerate
(Lemma~\ref{lnondeg}), we may define the
{\it triple product} of $x,y,z$ to be the unique element $xyz$ of $\g_1$
such that $q(w,x,y,z) = \gen{w,xyz}$ for all $w\in\g_1$.

Following Ferrar (\cite{ferrar}, \S3), we call a nonzero element $x\in\g_1$
{\it strictly regular} if $xxy\in Fx$ for all $y\in\g_1$.  In this section
we will give several equivalent characterizations of strictly regular elements.

\begin{lem}\label{lalphasr}
The basis element $x_\alpha$ is strictly regular.
\end{lem}
\begin{proof}
Let $\beta,\gamma$ be roots of $\alpha$-height~$1$.  By Corollary~\ref{ctworho}, if
$\gen{x_\gamma,x_\alpha x_\alpha x_\beta} = q(x_\gamma,x_\alpha,x_\alpha,x_\beta)$
is nonzero, then $2\alpha + \beta + \gamma = 2\rho$.
Since the simple root $\alpha$ has height~$1$, this implies
$\rootht(\beta + \gamma) = 2\rootht\rho - 2$.  As $\rho$ is the unique
highest root, $\beta$ and $\gamma$ have smaller heights than $\rho$, so this
can only occur if both have height $\rootht \rho - 1$.
Since the only simple root not orthogonal to $\rho$ is $\alpha$, the only root of
that height is $\rho-\alpha$, and $\gen{x_\gamma, x_\alpha x_\alpha x_\beta}$ is
therefore zero unless $\beta=\gamma=\rho-\alpha$.
The orthogonal complement of any $x_\alpha x_\alpha y$ thus includes the
space generated by all the $x_\gamma$, $\gamma\ne\rho-\alpha$.  Since this
is the orthogonal complement of $x_\alpha$, we have $x_\alpha x_\alpha y \in F x_\alpha$.
\end{proof}

\begin{cor}\label{cbetasr}
For any long root $\beta$ of $\alpha$-height~$1$, $x_\beta$ is strictly regular.
\end{cor}
\begin{proof}
Since the property of being strictly regular depends only on the triple
product, which is in turn defined in terms of the quartic and bilinear forms,
it is preserved by the action of $\goss$ by Lemma~\ref{lgoss}.  It
is also preserved by scaling, so it is preserved by the action of $\go$.
By Lemma~$2.1$ in \cite{rohrle}, all the elements $x_\beta$ with $\beta$ a long root
of $\alpha$-height~$1$ are in the same $\go$-orbit, so they are are all strictly
regular since $x_\alpha$ is.
\end{proof}

\begin{lem}\label{lxxy}
Let $x\in\g_1$ be such that $xxy=0$ for all $y\in\g_1$; then $x=0$.
\end{lem}
\begin{proof}
The set of all $x$ such that $x x \g_1=\{0\}$ is invariant under the action
of $\go$ on $\g_1$, so it is a union of $\go$-orbits; it is also closed
(in the Zariski topology).
Thus it suffices to show that $x x \g_1\ne\{0\}$ for a representative $x$ of the
smallest nonzero orbit (i.e., orbit~$1$); this follows if there are
$y,z\in\g_1$ such that $q(x,x,y,z)\ne 0$.  A representative of the smallest
nonzero orbit is $x=x_\alpha$; we let $y=z=x_{\rho-\alpha}$.
By \eqref{eqcomp1}, we have $q(x,x,y,z) = 1$.
\end{proof}

\begin{lem}\label{lorthotworho}
If $\beta_1,\beta_2,\beta_3,\beta_4$ are mutually orthogonal roots of $\alpha$-height~$1$,
then $\beta_1+\beta_2+\beta_3+\beta_4=2\rho$.
\end{lem}

\noindent This is Corollary 1.4 in \cite{rohrle}.

\begin{proof}
Since $\beta_1$ has $\alpha$-height~$1$, $\rho-\beta_1$ is a root.  Since $\beta_2$ is orthogonal
to $\beta_1$, $\gen{\rho-\beta_1,\beta_2} = \gen{\rho,\beta_2}-\gen{\beta_1,\beta_2} =
\gen{\rho,\beta_2}=1$,
so $\rho-\beta_1-\beta_2$ is also a root.  Continuing in this fashion, we find that
$\rho-\beta_1-\beta_2-\beta_3-\beta_4$ is a root; since it has $\alpha$-height~$-2$, it
must be $-\rho$.
\end{proof}

\begin{lem}\label{lortholong}
If four roots of $\alpha$-height~$1$ are mutually orthogonal, then they are all
long roots.
\end{lem}
\begin{proof}
Call the roots $\beta_1,\beta_2,\beta_3,\beta_4$.  By Lemma~\ref{lorthotworho},
$\beta_1+\beta_2+\beta_3+\beta_4=2\rho$; since the roots are mutually
orthogonal we then have
\begin{align*}
  4(\rho,\rho) &= (2\rho,2\rho)\\
	       &= (\beta_1+\beta_2+\beta_3+\beta_4,\beta_1+\beta_2+\beta_3+\beta_4)\\
               &= (\beta_1,\beta_1)+(\beta_2,\beta_2)+(\beta_3,\beta_3)+(\beta_4,\beta_4).
\end{align*}
Since $\rho$ is long, $(\beta_i,\beta_i)\le(\rho,\rho)$ for each $i$, $1\le i\le 4$; thus we must
have $(\beta_i,\beta_i)=(\rho,\rho)$ for each $i$.
\end{proof}

\begin{lem}\label{labcd}
Let $\alpha,\beta,\gamma,\delta$ be mutually orthogonal roots of
$\alpha$-height~$1$; then $q(x_\alpha,x_\beta,x_\gamma,x_\delta)\ne 0$.
\end{lem}
\begin{proof}
Since $x_\alpha+x_\beta+x_\gamma+x_\delta$ is a representative of the
dense orbit and $q$ is not identically zero,
$q(x_\alpha+x_\beta+x_\gamma+x_\delta)\ne 0$.
Expanding the corresponding $4$-linear form, we obtain five kinds of terms:
\begin{itemize}
\item Those with four equal arguments, e.g., $q(x_\beta, x_\beta, x_\beta, x_\beta)$.
Since we cannot have $4\beta=2\rho$, this expression
is zero by Lemma~\ref{lqsum}.
\item Those with exactly three equal arguments, e.g., $q(x_\beta, x_\beta, x_\beta, x_\gamma)$.
The mutually orthogonal roots $\alpha,\beta,\gamma,\delta$ are long
by Lemma~\ref{lortholong}.  Thus $x_\beta$ is strictly
regular (Corollary~\ref{cbetasr}), so the $4$-linear form here is
$\gen{x_\gamma,x_\beta x_\beta x_\beta} = \lambda\gen{x_\gamma,x_\beta}$
for some $\lambda\in F$; but $\gen{x_\gamma,x_\beta}=0$
because $\gamma+\beta$ is not a root.  Thus these terms are also zero.
\item Those with two pairs of equal arguments, e.g., $q(x_\beta, x_\beta, x_\gamma, x_\gamma)$.
Since $\beta+\gamma$ is not a root, it is not $\rho$.  Thus
$2\beta+2\gamma\ne 2\rho$, so this expression is zero.
\item Those with exactly two equal arguments, e.g., $q(x_\beta, x_\beta, x_\gamma, x_\delta)$.
By Lemma~\ref{lorthotworho}, $\alpha+\beta+\gamma+\delta=2\rho$; thus
$2\beta+\gamma+\delta\ne 2\rho$, so these terms are zero.
\item Those with four unequal arguments, e.g., $q(x_\alpha, x_\beta, x_\gamma, x_\delta)$,
which by elimination must be nonzero.\qedhere
\end{itemize}
\end{proof}

\begin{prop}\label{porbit1}
The strictly regular elements of $\g_1$ are those contained in the smallest nonzero
orbit.
\end{prop}
\begin{proof}
The set of strictly regular elements is a union of orbits; its union with $0$ is a
closed set.  Since $x_\alpha$ is a representative of the smallest nonzero orbit
and is strictly regular by Lemma~\ref{lalphasr}, all elements
of the smallest nonzero orbit are also strictly regular.  It thus suffices to show
that representatives of level~$2$ orbits are not strictly regular.
Let $\alpha,\beta,\gamma,\delta$ be four mutually orthogonal roots of
$\alpha$-height~$1$.  We take $x_\alpha + x_\beta$ as a representative of a level~$2$ orbit.

We compute
\begin{align*}
\gen{x_\delta,(x_\alpha + x_\beta)(x_\alpha + x_\beta)x_\gamma} &=
  q(x_\alpha + x_\beta,x_\alpha + x_\beta,x_\gamma,x_\delta)\\
  &= q(x_\alpha,x_\alpha,x_\gamma,x_\delta)
    +2q(x_\alpha,x_\beta,x_\gamma,x_\delta)\\
  &\quad  +q(x_\beta,x_\beta,x_\gamma,x_\delta)\\
  &=2q(x_\alpha,x_\beta,x_\gamma,x_\delta),
\end{align*}
the other terms being zero since $\alpha+\alpha+\gamma+\delta$ and
$\beta+\beta+\gamma+\delta$ cannot equal $2\rho$ since
$\alpha+\beta+\gamma+\delta=2\rho$ by Lemma~\ref{lorthotworho}.
By Lemma~\ref{labcd}, the result is nonzero, so in particular the triple product
$(x_\alpha + x_\beta)(x_\alpha + x_\beta)x_\gamma$ is
not orthogonal to $x_\delta$.
However, $\gen{x_\alpha+x_\beta, x_\delta} = 0$ since neither
$\alpha+\delta$ nor $\beta+\delta$ is a root.  Hence the triple
product $(x_\alpha + x_\beta)(x_\alpha + x_\beta)x_\gamma$ is not
a scalar multiple of $x_\alpha+x_\beta$; thus $x_\alpha+x_\beta$ is
not strictly regular.
\end{proof}

\begin{lem}\label{lspan}
The strictly regular elements span $\g_1$.
\end{lem}
\begin{proof}
By Proposition~\ref{porbit1}, orbit~$1$ consists of strictly regular elements.
The span of orbit~$1$ is invariant under the action of $\go$; thus it is
a union of orbits.  Both $x_\alpha$ and $x_{\rho-\alpha}$ are in
orbit~$1$, so $x_\alpha + x_{\rho-\alpha}$ is in their span, but is
also a representative of the dense orbit.  Thus all of the dense orbit is
in the span of orbit~$1$.  Since the dense orbit is not contained in a
proper subspace, the span of orbit~$1$ is all of $\g_1$.
\end{proof}

An element $x\in\g_1$ is {\it rank one} if $x x \g_1$ is a one-dimensional
vector space over $F$.

\begin{prop}\label{prank1}
An element $x\in\g_1$ is strictly regular if and only if it is rank one.
\end{prop}
\begin{proof}
Suppose $x$ is strictly regular.  By definition,  $x x \g_1$ is contained in the
one-dimensional space $Fx$.  In the case $x=x_\alpha$, we know $x_\alpha x_\alpha \g_1$
is not zero because $\gen{x_{\rho-\alpha}, x_\alpha x_\alpha x_{\rho-\alpha}} =
q(x_{\rho-\alpha}, x_\alpha, x_\alpha, x_{\rho-\alpha})$, which is~$1$ by \eqref{eqcomp1}.
The condition that $x x \g_1$ is not zero is invariant under the action of $\go$, so it holds for all
of orbit~$1$.

As in the proof of the previous proposition, let $\alpha,\beta,\gamma,\delta$
be four mutually orthogonal roots of $\alpha$-height~$1$, and choose
$x = x_\alpha + x_\beta$ as a representative of a level~$2$ orbit.  Since the
set of rank one elements is a closed union of orbits, it will suffice
to show that $x$ is not rank one.  We have
$\gen{x_{\rho-\beta},x x x_{\rho-\alpha}}=2q(x_{\rho-\beta},x_\alpha,x_\beta,x_{\rho-\alpha})\ne0$,
by Corollary~\ref{ctworho} and \eqref{eqcomp2}.  However,
$\gen{x_{\rho-\beta},x x x_\gamma} =
q(x_{\rho-\beta},x_\alpha, x_\alpha, x_\gamma) +
q(x_{\rho-\beta},x_\beta, x_\beta, x_\gamma) +
2q(x_{\rho-\beta},x_\alpha,x_\beta,x_\gamma) = 0$,
where we know the first term is zero because it is
$\gen{x_{\rho-\beta},x_\alpha x_\alpha x_\gamma}$ and the triple product
is a scalar multiple of $x_\alpha$; the other two terms are zero
by Corollary~\ref{ctworho}.  On the other hand,
we know that $x x x_\gamma$ is nonzero since $\gen{x_\delta,x x x_\gamma}=
2q(x_\alpha,x_\beta,x_\gamma,x_\delta)$ which is not zero by
Lemma~\ref{labcd}.  Thus $xxx_{\rho-\alpha}$ and $xxx_\gamma$ do not
lie in the same one-dimensional subspace, so $x$ is not rank one.
\end{proof}

The following result allows us to compute the triple product and the
$4$-linear form if two of the arguments are the same strictly regular
element.

\begin{lem}\label{leq5}
For $x$ strictly regular and any $y,z\in\g_1$,
\begin{align}
x x y &= \gen{y,x}x,\label{eq5a}\\
q(x,x,y,z) &= \gen{y,x}\gen{z,x}.\label{eq5b}
\end{align}
\end{lem}
\begin{proof}
Since $x$ is strictly regular, for any $y\in\g_1$ we have $xxy\in Fx$.
If $\gen{y,x}=0$, then for any $z\in\g_1$ we have $\gen{z,xxy}=q(z,x,x,y)
=\gen{y,xxz}=0$, thus $xxy=0$.
Define $f : \g_1\to F$ by $xxy = f(y)x$; then $f$ is a linear form and
$f(y)$ is zero whenever $\gen{y,x}$ is zero.  Thus $f(-)$ is a
scalar multiple of $\gen{-,x}$.

By Proposition~\ref{porbit1}, $x$ is in orbit~$1$; by Lemma~\ref{lalphasr},
so is $x_\alpha$.  Hence there is some element $g\in\goss$ such that
$g\cdot x=c x_\alpha$ for some $c\in F^\times$.  Let
$x' = g^{-1}\cdot x_{\rho-\alpha}$; since the bilinear form is
preserved by the action of $\goss$ (Lemma~\ref{lgoss}), we have
$\gen{x',x} = \gen{x_{\rho-\alpha},c x_\alpha} = \pm c$.
We can now compute $q(x,x,x',x')$ in two ways.  On the one hand, since
the $4$-linear form is also preserved, we have
\begin{align*}
q(x,x,x',x')&=q(c x_\alpha, c x_\alpha, x_{\rho-\alpha},x_{\rho-\alpha})\\
            &=c^2 q(x_\alpha, x_\alpha, x_{\rho-\alpha},x_{\rho-\alpha})\\
            &=c^2\qquad\qquad(\text{by \eqref{eqcomp1}})\\
            &=\gen{x',x}^2.
\end{align*}
On the other hand, it is $\gen{x',xxx'}=\gen{x',f(x')x}=f(x')\gen{x',x}$.
Thus $f(x')=\gen{x',x}$, and therefore $f(y) = \gen{y,x}$ for any $y\in\g_1$.

By the definition of $f$, we now have $x x y = \gen{y,x}x$ for all $y\in\g_1$.
Further, for any $z\in\g_1$ we have $q(x,x,y,z)=\gen{z,xxy}=\gen{y,x}\gen{z,x}$.
\end{proof}

\begin{lem}\label{lunique}
Each element in the dense orbit of $\g_1$ can be expressed as the
sum of two strictly regular elements in one and only one way.
\end{lem}
\begin{proof}
Since the action of $\goss$ and scaling by elements of $F^\times$ both preserve
strictly regular elements, it suffices to prove this for any representative of
the dense orbit.  We choose $x = x_\alpha + x_{\rho-\alpha}$ as
the representative, which establishes the existence of such an expression.

Suppose $x = u + v$ with $u,v$ strictly regular.  The triple product $xxx$ is thus
\begin{align*}
(u + v)(u + v)(u + v) &= uuu + 3uuv + 3uvv + vvv\\
  &= \gen{u,u}u + 3\gen{v,u}u + 3\gen{u,v}v + \gen{v,v}v\\
  &= 3\gen{v,u}(u - v);
\end{align*}
in particular, this is true if $u=x_\alpha$ and $v=x_{\rho-\alpha}$, so we
have shown that
\begin{equation}\label{equv}
3\gen{v,u}(u-v) = 3\gen{x_{\rho-\alpha},x_\alpha}(x_\alpha - x_{\rho-\alpha}).
\end{equation}
The quartic form $q(x)=\gen{x,xxx}$ is thus
\begin{align*}
\gen{u + v, 3\gen{v,u}(u - v)} &= 3\gen{v,u}(-\gen{u,v}+\gen{v,u})\\
  &= 6\gen{v,u}^2;
\end{align*}
again, this must be the same as $6\gen{x_{\rho-\alpha},x_\alpha}^2$.  Thus
$\gen{v,u}=\pm\gen{x_{\rho-\alpha},x_\alpha}$, so \eqref{equv}
yields $u - v = \pm(x_\alpha - x_{\rho-\alpha})$.  Combined with
$u + v = x_\alpha + x_{\rho-\alpha}$, one choice of sign yields
$u = x_\alpha$, $v = x_{\rho-\alpha}$, and the other $u = x_{\rho-\alpha}$,
$v=x_\alpha$, so the choice of $u$ and $v$ is determined up to order.
\end{proof}

\section{Freudenthal triple systems}\label{secfts}

A {\it Freudenthal triple system} is a finite-dimensional vector space $V$
over a field~$F$ (with characteristic not $2$ or $3$) such that
\begin{itemize}
\item There is a nonzero quartic form $q$ defined on $V$.  A corresponding
$4$-linear form, also called $q$, is given by linearization, with $q(x,x,x,x) = q(x)$
for all $x\in V$.
\item There is a nondegenerate skew-symmetric bilinear form $\gen{-,-}$ defined
on $V$.  Thus for given $x,y,z\in V$ we may define the triple product $xyz$ to
be the unique vector in $V$ such that $q(w,x,y,z) = \gen{w,xyz}$ for all $w\in V$.
\item The triple product satisfies the following identity:
\begin{equation}
2(xxx)xy = \gen{y,x}xxx + \gen{y,xxx}x.\label{eqtriple}
\end{equation}
\end{itemize}

Definitions of Freudenthal triple system in the literature vary.
For example, in \cite{ferrar} the $2$ on the left-hand side of \eqref{eqtriple} is omitted;
in \cite{springer} the $2$ becomes a $6$ and the triple product
is defined so that $8 q(w,x,y,z) = \gen{xyz,w}$.  However, these
variations are inessential; it is easy to convert one definition to
another by rescaling the quartic and bilinear forms as needed.

\begin{thm}\label{tFTS}
The vector space $\g_1$ equipped with the quartic form $q$ and the bilinear form
$\gen{-,-}$ is a Freudenthal triple system.
\end{thm}
\begin{proof}
We established in Section~\ref{secforms} that $\gen{-,-}$ is skew-symmetric and
nondegenerate and that $q$ is nonzero.  Hence it remains only to show that the
triple product identity \eqref{eqtriple} is satisfied.

We first set $x = x_\alpha+x_{\rho-\alpha}$.  As in the proof of
Lemma~\ref{lunique}, we use \eqref{eq5a} to compute
$xxx = 3\gen{x_{\rho-\alpha},x_\alpha}(x_\alpha - x_{\rho-\alpha})$.
Thus the left-hand side of \eqref{eqtriple} is
\begin{align*}
2(xxx)xy &= 6\gen{x_{\rho-\alpha},x_\alpha}(x_\alpha - x_{\rho-\alpha})(x_\alpha + x_{\rho-\alpha})y\\
  &= 6\gen{x_{\rho-\alpha},x_\alpha}(x_\alpha x_\alpha y - x_{\rho-\alpha}x_{\rho-\alpha}y)\\
  &= 6\gen{x_{\rho-\alpha},x_\alpha}(\gen{y,x_\alpha} x_\alpha - \gen{y,x_{\rho-\alpha}}x_{\rho-\alpha}).
\end{align*}
The right-hand side is
\begin{align*}
\gen{y,x}xxx+\gen{y,xxx}x &= 3\gen{x_{\rho-\alpha},x_\alpha}
     (\gen{y,x_\alpha} + \gen{y,x_{\rho-\alpha}})(x_\alpha - x_{\rho-\alpha})\\
  &\quad + 3\gen{x_{\rho-\alpha},x_\alpha}
     (\gen{y,x_\alpha} - \gen{y,x_{\rho-\alpha}})(x_\alpha + x_{\rho-\alpha})\\
  &= 6\gen{x_{\rho-\alpha},x_\alpha}(\gen{y,x_\alpha} x_\alpha - \gen{y,x_{\rho-\alpha}}x_{\rho-\alpha});
\end{align*}
thus \eqref{eqtriple} holds for $x=x_\alpha+x_{\rho-\alpha}$ and any $y\in\g_1$.

Since the action of $\goss$ on $\g_1$ stabilizes the bilinear form and the
triple product, and since \eqref{eqtriple} is preserved if $x$ is adjusted
by a scalar factor, it holds for the entire orbit of $x$, which is the dense orbit.
Since the identity is a polynomial condition it also holds on the closure,
which is all of $\g_1$.
\end{proof}

\section{Computation of the $4$-linear form}\label{secq}

In this section we show how to evaluate the expression
$q(x_\beta,x_\gamma,x_\delta,x_\epsilon)$
whenever $\beta,\gamma,\delta,\epsilon$ are {\it long} roots of $\alpha$-height~$1$.
Among the Lie algebras we are considering, the roots are always long in
types $D$ and $E$, so, by linearity, this will suffice to
compute $q$ for any values in $\g_1$ in these cases.

\begin{lem}
Suppose $\beta_1,\beta_2,\beta_3,\beta_4$ are long roots of $\alpha$-height~$1$
and that their sum is $2\rho$.  It follows that
\begin{equation}\label{eqb1}
\gen{\beta_1,\beta_2} + \gen{\beta_1,\beta_3} + \gen{\beta_1,\beta_4} = 0
\end{equation}
and
\begin{equation}\label{eqb2}
\gen{\beta_1,\beta_2} = \gen{\beta_3,\beta_4}.
\end{equation}
\end{lem}
\begin{proof}
We may reverse the arguments of $\gen{-,-}$
whenever both are long roots.  Thus to show \eqref{eqb1} we compute
$\gen{\beta_1,\beta_2} + \gen{\beta_1,\beta_3} + \gen{\beta_1,\beta_4} =
\gen{\beta_2,\beta_1} + \gen{\beta_3,\beta_1} + \gen{\beta_4,\beta_1} =
\gen{2\rho-\beta_1,\beta_1} = 2\gen{\rho,\beta_1}-\gen{\beta_1,\beta_1} = 0$.

To show \eqref{eqb2}, we expand the equal expressions
$(\beta_1+\beta_2,\beta_1+\beta_2)$ and
$(2\rho-\beta_3-\beta_4,2\rho-\beta_3-\beta_4)$.
Taking the long roots to have unit length, we have on the one hand
$(\beta_1+\beta_2,\beta_1+\beta_2) = 2 + 2(\beta_1,\beta_2)$.
Keeping in mind that, for example, $2(\rho,\beta_3)=\gen{\rho,\beta_3}=1$,
we have on the other hand
\begin{align*}
(2\rho-\beta_3-\beta_4,2\rho-\beta_3-\beta_4) &=
   6 - 4(\rho,\beta_3) - 4(\rho,\beta_4) + 2(\beta_3,\beta_4)\\
   &= 2 + 2(\beta_3,\beta_4).
\end{align*}
Thus $2(\beta_1,\beta_2)=2(\beta_3,\beta_4)$; that is,
$\gen{\beta_1,\beta_2} = \gen{\beta_3,\beta_4}$.
\end{proof}

\begin{prop}\label{psumtworho}
If the sum of four long roots of $\alpha$-height~$1$ is $2\rho$, then one
of the following three cases must hold:
\begin{enumerate}
\item[$(a)$] The four roots consist of two equal pairs; that is, they are of
the form $\beta,\beta,\rho-\beta,\rho-\beta$ for some $\beta$.
\item[$(b)$] The four roots consist of distinct pairs that sum to $\rho$;
that is, they are of the form $\beta,\rho-\beta,\gamma,\rho-\gamma$
for distinct $\beta,\gamma$.  Moreover, we may take $\beta$ and $\gamma$
to be orthogonal.
\item[$(c)$] The four roots are mutually orthogonal.
\end{enumerate}
\end{prop}
\begin{proof}
Let $\beta_1,\beta_2,\beta_3,\beta_4$ be four such roots.  No two can
be opposite since all have $\alpha$-height~$1$.  If any two are equal,
say $\beta_1=\beta_2$, then by~\eqref{eqb2} we have $2=\gen{\beta_1,\beta_2}=
\gen{\beta_3,\beta_4}$, so $\beta_3=\beta_4$ as well.  This is case (a).

Suppose some root, say $\beta_1$, is not orthogonal to all of the others.
By \eqref{eqb1} we have
$\gen{\beta_1,\beta_2} + \gen{\beta_1,\beta_3} + \gen{\beta_1,\beta_4} = 0$;
since each term is $-1$, $0$ or $1$ and not all are zero, we must have one of each.
Without loss of generality, assume $\gen{\beta_1,\beta_2}=-1$ and
$\gen{\beta_1,\beta_3}=0$; then $\beta_1+\beta_2$
is a root.  Since it has $\alpha$-height~$2$, it must be $\rho$.   By \eqref{eqb2},
we also have $\gen{\beta_3,\beta_4}=-1$, thus also $\beta_3+\beta_4=\rho$.
Thus we are in case (b).  As indicated, we have $\beta_1$ and $\beta_3$
orthogonal.

The only remaining possibility is that the four roots are mutually orthogonal,
which is case (c).
\end{proof}

We now proceed to give the value of $q(\beta_1,\beta_2,\beta_3,\beta_4)$ in
each of the three cases.  We remind the reader that we will be
making extensive use of the facts about structure constants previously
mentioned in Section~\ref{secforms}.

The first case was already handled in Lemma~\ref{lcase1}, where we showed that
$q(x_\beta,x_\beta,x_{\rho-\beta},x_{\rho-\beta}) = 1$ for any long root $\beta$
of $\alpha$-height~$1$.  The second case
was computed in Lemma~\ref{lcase2}; there we found
$q(x_\beta,x_\gamma,x_{\rho-\beta},x_{\rho-\gamma}) = -\frac12 c_{\beta,-\rho} c_{\gamma,-\rho}$
where $\beta$ and $\gamma$ are orthogonal long roots of $\alpha$-height~$1$.
The remaining case is covered by the following lemma.

\begin{lem}
If $\beta_1,\beta_2,\beta_3,\beta_4$ are mutually orthogonal roots of $\alpha$-height~$1$,
then
\begin{equation*}
q(x_{\beta_1},x_{\beta_2},x_{\beta_3},x_{\beta_4}) = 
  c_{\beta_1,\beta_4-\rho}c_{\beta_2,\beta_1-\rho}
  c_{\beta_3,\beta_4-\rho}c_{\beta_4,\beta_1-\rho}
  \ne 0.
\end{equation*}
\end{lem}
\begin{proof}
By Lemma~\ref{lorthotworho}, the sum of four mutually orthogonal roots of $\alpha$-height~$1$
is $2\rho$, and by Lemma~\ref{lortholong} they are all long roots.  We will
apply $\eqref{eqcarter4}$ with $\beta,\gamma,\delta,\epsilon = \beta_1,
\beta_2,\beta_3-\rho,\beta_4-\rho$.  Observe that $\beta+\gamma+\delta+\epsilon=0$
and no two of $\beta,\gamma,\delta,\epsilon$ are opposite; for example,
$\beta+\delta=0$ implies $\beta_1+\beta_3=\rho$, but $\beta_1$ and $\beta_3$ are
orthogonal.  With these values, $\eqref{eqcarter4}$ becomes
\[
	 c_{\beta_1,\beta_2} c_{\beta_3-\rho,\beta_4-\rho} +
	 c_{\beta_2,\beta_3-\rho} c_{\beta_1,\beta_4-\rho} +
   c_{\beta_3-\rho,\beta_1} c_{\beta_2,\beta_4-\rho} = 0.
\]
The structure constants in the first term are zero since $\beta_1+\beta_2$ is
not a root.
Since $\beta_2+\beta_3-\rho$ and $\beta_1+\beta_3-\rho$ are roots the
remaining terms are not zero.

We now have
$c_{\beta_2,\beta_3-\rho} c_{\beta_1,\beta_4-\rho} = -c_{\beta_3-\rho,\beta_1} c_{\beta_2,\beta_4-\rho}$.
Using $a_{ij}$ as an abbreviation for $c_{\beta_i,\beta_j-\rho}$, we can
rewrite this as
\begin{equation}\label{eq3}
a_{23} a_{14} = a_{13} a_{24}.
\end{equation}
Since the numbering of the indices is arbitrary, we think of this as saying
that, in a product of the form $a_{ij}a_{kl}$ that uses four different
indices, we may interchange the first subscripts of the two factors.

Since all the $a_{ij}$ are $\pm1$, we can freely move them across the
equals sign; in particular, we also have
\begin{equation}\label{eq2}
a_{13} a_{23} = a_{14} a_{24};
\end{equation}
in other words, in a product of the form $a_{ij}a_{kj}$ involving three
different indices, the repeated index may be replaced by the unused one.

A typical term in the sum for $q(x_{\beta_1},x_{\beta_2},x_{\beta_3},x_{\beta_4})$
given by Lemma~\ref{lqsum} is
\begin{align*}
c_{-\rho,\beta_1}c_{\beta_1-\rho,\beta_2}c_{\rho-\beta_3-\beta_4,\beta_3}
c_{\rho-\beta_4,\beta_4} &=
c_{\beta_1,-\rho}c_{\beta_2,\beta_1-\rho}c_{\beta_3,\beta_4-\rho}
c_{\beta_4,-\rho}\\
&= c_{\beta_1,\beta_4-\rho}c_{\beta_2,\beta_1-\rho}c_{\beta_3,\beta_4-\rho}
c_{\beta_4,\beta_1-\rho}\\
&= a_{14}a_{21}a_{34}a_{41},
\end{align*}
where we have used Lemma~\ref{laux2} for the second equality.  Every other term in
the sum is obtained by permuting the indices; we will show that the value
is unchanged in each case.  Since the two permutations given by
$1\mapsto2\mapsto3\mapsto4\mapsto1$ and by $1\mapsto2\mapsto1$ generate the
symmetric group, it suffices to show that $a_{21}a_{32}a_{41}a_{12}$ and
$a_{24}a_{12}a_{34}a_{42}$ are the same as the product above.

We first apply the principle of \eqref{eq2} in the form
$a_{14}a_{34}=a_{12}a_{32}$ to find that
$a_{14}a_{21}a_{34}a_{41} = a_{12}a_{21}a_{32}a_{41} = a_{21}a_{32}a_{41}a_{12}$,
so the first required equality holds.  Proceeding from the last expression, we
alternately apply \eqref{eq2} and \eqref{eq3} as follows:
\begin{align*}
a_{21}a_{32}a_{41}a_{12}
&=a_{23}a_{32}a_{43}a_{12},\quad(\text{since }a_{21}a_{41}=a_{23}a_{43})\\
&=a_{23}a_{32}a_{13}a_{42},\quad(\text{since }a_{43}a_{12}=a_{13}a_{42})\\
&=a_{24}a_{32}a_{14}a_{42},\quad(\text{since }a_{23}a_{13}=a_{24}a_{14})\\
&=a_{24}a_{12}a_{34}a_{42},\quad(\text{since }a_{32}a_{14}=a_{12}a_{34})
\end{align*}
which is the required product.

Thus all $24$ summands are equal, so we have
\[
q(x_{\beta_1},x_{\beta_2},x_{\beta_3},x_{\beta_4})=a_{14}a_{21}a_{34}a_{41},
\]
which, by substituting for the $a_{ij}$, becomes the desired equation.
\end{proof}

To summarize, we have the following result.

\begin{prop}\label{pq}
If $\beta_1,\beta_2,\beta_3,\beta_4$ are long roots of $\alpha$-height~$1$, then the
value of $q(x_{\beta_1},x_{\beta_2},x_{\beta_3},x_{\beta_4})$ is one of the following:
\begin{itemize}
\item $0$, if $\beta_1+\beta_2+\beta_3+\beta_4\ne2\rho$;
\item $1$, if $\beta_1+\beta_2+\beta_3+\beta_4=2\rho$ and there are two
pairs of equal roots;
\item $-\frac12 c_{\beta,-\rho} c_{\gamma,-\rho}$ if the roots are, in
some order, $\beta,\gamma,\rho-\beta,\rho-\gamma$ with $\gen{\beta,\gamma}=0$ for
some $\beta,\gamma$; or
\item $c_{\beta_1,\beta_4-\rho}c_{\beta_2,\beta_1-\rho}
       c_{\beta_3,\beta_4-\rho}c_{\beta_4,\beta_1-\rho}$ if the four roots are mutually orthogonal.
\end{itemize}
In particular, $q(x_{\beta_1},x_{\beta_2},x_{\beta_3},x_{\beta_4})$ is nonzero
whenever $\beta_1+\beta_2+\beta_3+\beta_4=2\rho$.
\end{prop}

\section{Eigenspace decomposition of $\g_1$}\label{seceigen}

In this section we assume that $\g$ is a Lie algebra of type $D$ or $E$.
We show that there is an element $h$ in the torus~$\h$
such that $\g_1$ is the direct sum of the four eigenspaces under $\ad h$ corresponding to
the eigenvalues $-3,-1,1,3$, and that the eigenspaces corresponding to the eigenvalues
$-3$ and $3$ are one-dimensional (cf.~\cite{ferrar}, \S4).  This is a consequence of
the following proposition about the corresponding root systems.

\begin{prop}\label{peigen}
Let $\Psi$ be a root system of type $D$ or~$E$.  For any root $\beta\in\Psi$
of $\alpha$-height~$1$ we have
\[
\gen{\rho-2\alpha,\beta}=
\begin{cases}
-3&\text{\rm if }\beta=\alpha,\\
3&\text{\rm if }\beta=\rho-\alpha,\\
\pm1&\text{\rm otherwise.}
\end{cases}
\]
Moreover, the cases $\gen{\rho-2\alpha,\beta}=-1$ and $\gen{\rho-2\alpha,\beta}=1$
occur equally often.
\end{prop}
\begin{proof}
Let $\beta$ be a root of $\alpha$-height~$1$.  For each such root, $\rho-\beta$
is another root of $\alpha$-height~$1$, and we have
$\gen{\alpha,\beta}+\gen{\alpha,\rho-\beta} = 1$.  Since $\gen{\alpha,\beta}=2$
only if $\beta=\alpha$, it follows that $\gen{\alpha,\beta}=-1$ only if $\beta=\rho-\alpha$.
Thus for the remaining pairs of roots $\beta$, $\rho-\beta$ we have
$\gen{\alpha,\beta}=0$ or $1$ and correspondingly $\gen{\alpha,\rho-\beta}=1$ or $0$.

As $\gen{\rho,\beta}=1$, we have $\gen{\rho-2\alpha,\beta}=1-2\gen{\alpha,\beta}$.
Thus $\gen{\rho-2\alpha,\alpha}=-3$ and
$\gen{\rho-2\alpha,\rho-\alpha}=3$, with the remaining cases split equally
between $\gen{\rho-2\alpha,\beta}=1$ and $\gen{\rho-2\alpha,\beta}=-1$.
\end{proof}

The above proposition can be generalized by using $\rho-2\alpha'$ with $\alpha'$ any root
of $\alpha$-height~$1$ in place of $\rho-2\alpha$; the proof goes through unchanged.
However, we do not make use of this added generality.

At this point, we know that the promised element of $\h$ exists because
the Chevalley basis gives an isomorphism between $\h$ and the coroot lattice
with scalars extended to $F$.  To give it explicitly, recall that, for any
root $\beta$, the element $h_\beta\in\h$ is defined to be
$[x_\beta, x_{-\beta}]$ and has the property that
$[h_\beta,x_\gamma]=\gen{\gamma,\beta} x_\gamma$
for any root $\gamma$ (see \cite{humlie}, \S\S 8.3, 25.2).
Setting $h = h_{\rho-\alpha} - h_{\alpha}\in\h$, we then
have $[h,x_\beta] = (\gen{\beta,\rho-\alpha}-\gen{\beta,\alpha})x_\beta =
\gen{\rho-2\alpha,\beta}x_\beta$, yielding the eigenvalue decomposition described
above.

\section{Characterization of the orbits}\label{secorbits}

\begin{lem}\label{laab}
Let $\beta,\gamma$ be roots of $\alpha$-height~$1$.  The triple product
$x_\beta x_\beta x_\gamma$ is zero unless $\beta + \gamma = \rho$.
\end{lem}
\begin{proof}
Since $x_\beta$ is strictly regular (Corollary~\ref{cbetasr}),
\eqref{eq5a} gives $x_\beta x_\beta x_\gamma = \gen{x_\gamma, x_\beta} x_\beta$.
As $\gen{x_\gamma, x_\beta}$ is zero unless $\beta+\gamma=\rho$, the result follows.
\end{proof}

\begin{prop}\label{porbits}
In the cases where there are five $G_0$-orbits in $\g_1$, namely for
$\g$ of type $E_6$, $E_7$ or $E_8$, the orbits are characterized as follows:
\begin{itemize}
	\item $x$ is in orbit~$0$ iff $x=0$,
	\item $x$ is in the closure of orbit~$1$ iff $xxy\in Fx$ for all $y\in\g_1$,
	\item $x$ is in the closure of orbit~$2$ iff $xxx=0$,
	\item $x$ is in the closure of orbit~$3$ iff $q(x)=0$, and
	\item $x$ is in orbit~$4$ iff $q(x)\ne0$.
\end{itemize}
\end{prop}
\begin{proof}
The statement for orbit~$1$ is Proposition~\ref{porbit1}.

The conditions for orbits $2$ and $3$ are invariant under the action of $\go$
and define closed sets, so it suffices to consider representatives of the orbits.
Let $\beta_1,\beta_2,\beta_3,\beta_4$ be four mutually
orthogonal roots of $\alpha$-height~$1$.

Choose $x=x_{\beta_1}+x_{\beta_2}$ as a representative of orbit~$2$.  The triple
product $xxx$ contains the terms $x_{\beta_1} x_{\beta_1} x_{\beta_1}$,
$x_{\beta_2} x_{\beta_2} x_{\beta_2}$, $x_{\beta_1} x_{\beta_1} x_{\beta_2}$ and
$x_{\beta_1} x_{\beta_2} x_{\beta_2}$.  All are zero by Lemma~\ref{laab}; thus
$xxx=0$.

Conversely, for $x=x_{\beta_1}+x_{\beta_2}+x_{\beta_3}$ in orbit~$3$, we have
$xxx=6 x_{\beta_1}x_{\beta_2}x_{\beta_3}$ since the other terms vanish by
Lemma~\ref{laab}.  Thus we have
$\gen{x_{\beta_4},xxx}=6 q(x_{\beta_1},x_{\beta_2},x_{\beta_3},x_{\beta_4})$,
which is not zero by Lemma~\ref{labcd}.
Hence $xxx\ne 0$.

For $x=x_{\beta_1}+x_{\beta_2}+x_{\beta_3}$ in orbit~$3$, all the terms arising when
$q(x,x,x,x)$ is expanded are zero:  some $x_{\beta_i}$ must be repeated, so we have
terms of the form $q(x_{\beta_i},x_{\beta_i},x_{\beta_j},x_{\beta_k})$ with $i,j,k$
not necessarily distinct; such a term equals
$\gen{x_{\beta_j},x_{\beta_i}x_{\beta_i}x_{\beta_k}}$, which is $0$ by Lemma~\ref{laab}.

Finally, the fourth orbit is represented by $x=x_\alpha + x_{\rho-\alpha}$
(\cite{rohrle}, Cor. 4.4).  By the remark following Lemma~\ref{lcase1}, we have
$q(x) = 6$; hence $q(x)\ne 0$ for any $x$ in orbit~$4$.
\end{proof}

A similar result applies for Lie algebras of type $D_n$, except that the elements
$x\in\g_1$ satisfying $xxx=0$ are those that belong to any of the level~$2$ orbits
or their closures.  As these orbits are each
represented by elements of the form $x_{\beta_1}+x_{\beta_2}$, but for different
choices of $\beta_1,\beta_2,\beta_3,\beta_4$, the proof goes through unchanged.

Krutelevich (\cite{krutelevich}, Definition $22$) defines the {\it rank} of
an element of Freudenthal triple system constructed from a cubic Jordan algebra
using characterizations which are nearly the same as those given in the
preceding proposition.  His definition of rank~$1$ differs from the
characterization of orbit~$1$; it is equivalent (apart from a different
convention on scalars) to \eqref{eq5a}.

\section{Related groups}\label{secgroups}

As in Ferrar (\cite{ferrar}, \S7), we define two subgroups of the
group of linear automorphisms of $\g_1$.  The first, $Q$, preserves
the quartic form on $\g_1$ up to a nonzero scalar factor, that is,
\[
Q = \{\eta\in\GL(\g_1) : \forall x\in \g_1, q(\eta(x)) = r q(x) \text{ for some }
r\in F^\times\}.
\]
We call $r$ the {\it ratio of $\,\eta$ in $Q$}.

Similarly, the elements of $B$ are those that preserve the bilinear form
up to a nonzero scalar:
\[
B = \{\eta\in\GL(\g_1) : \forall x,y\in \g_1, \gen{\eta(x),\eta(y)} = r \gen{x,y} \text{ for some } r\in F^\times\}.
\]
In this case, we call $r$ the {\it ratio of $\,\eta$ in $B$}.

\begin{lem}\label{lsrinvt}
The set of strictly regular elements is invariant under any $\eta\in\GL(\g_1)$
that preserves the quartic form.
\end{lem}

\noindent The following argument is adapted from Ferrar (\cite{ferrar}, Cor. 7.2).

\begin{proof}
Suppose $x\in\g_1$ is rank one; then $q(x,x,y,z) = \gen{z,xxy}$
is zero for all $y\in \g_1$ and all $z$ in a codimension-$1$ subspace.
Conversely, if $x\ne0$ and $q(x,x,y,z) = \gen{z,xxy}$ is zero for all $y\in\g_1$
and all $z$ in a codimension-$1$ subspace, then $xx\g_1$ lies in a 
$1$-dimensional space.  Since $xx\g_1$ is not zero (Lemma~\ref{lxxy}),
$x$ is rank one.  Thus this condition on the $4$-linear form characterizes the
rank one elements among the nonzero elements of $\g_1$.

Since any $\eta$ in $\GL(\g_1)$ is nonsingular, it preserves the dimension of
subspaces.  If $\eta$ preserves the quartic form (and hence the $4$-linear
form), then the condition on the $4$-linear form is true of $\eta(x)$ if it is
for $x$.  Thus $\eta$ maps rank one elements to rank one elements; by
Proposition~\ref{prank1}, it thus maps strictly regular
elements to strictly regular elements.
\end{proof}

\begin{prop}\label{pgroups}
$Q$ is a subgroup of $B$.
\end{prop}
\begin{proof}
Let $\eta$ be an element of $Q$.
To show that $\eta$ preserves $\gen{x,y}$ up to a scalar factor, it suffices
to show it for all $x$ in a spanning set, such as the strictly regular
elements (Lemma~\ref{lspan}), and all $y\in\g_1$.

By \eqref{eq5b}, for $x$ strictly regular and any $y\in\g_1$ we have
$q(x,x,y,y) = \gen{x,y}^2$.  By Lemma~\ref{lsrinvt}, $\eta(x)$ is also strictly regular,
so
\[
\gen{\eta(x),\eta(y)}^2 = q(\eta(x),\eta(x),\eta(y),\eta(y))
=r\cdot q(x,x,y,y) = r \gen{x,y}^2,
\]
where $r$ is the ratio of $\eta$ in $Q$.
Thus $r$ is a square, say $r=s^2$; we then have $\gen{\eta(x),\eta(y)} = \pm s \gen{x,y}$.
The choice of sign does not depend on $y$, since for any $y_1,y_2\in\g_1$ we have
$\pm s \gen{x,y_1+y_2} = \gen{\eta(x),\eta(y_1+y_2)} = \pm s \gen{x,y_1} \pm s \gen{x,y_2}$,
so the signs must be the same whenever the bilinear forms are nonzero.
Let us say that $x$ is associated with $s$ if $\gen{\eta(x),\eta(y)} = s \gen{x,y}$
for all $y\in\g_1$, or that $x$ is associated with $-s$ otherwise.

The set of strictly regular elements associated to $s$ (resp., to $-s$) is a relatively
closed subset of the set of all strictly regular elements, and the set of strictly
regular elements is the disjoint union of these two sets.  However, since the set of
strictly regular elements is an orbit under the action of the connected set $\go$
(Proposition~\ref{porbit1}), it is connected.  Thus all strictly regular elements
are associated to the same square root of $r$. 
\end{proof}

\begin{cor}\label{corthoinvt}
Any element $\eta\in \GL(\g_1)$ that stabilizes the quartic form also preserves
orthogonality.
\end{cor}
\begin{proof}
If $\eta$ stabilizes the quartic form, it is in $Q$ (with ratio $1$);
thus it is in $B$ (with ratio $\pm 1$).
Therefore, for any $x,y\in \g_1$, we have $\gen{x,y}=0$ if and only if
$\gen{\eta(x),\eta(y)} = 0$.
\end{proof}

\section{The stabilizer of the quartic form: $G = E_8$}\label{secE8}

Suppose that $G$ is of type $E_8$ and $\g$ is thus the Lie algebra $E_8$,
which has dimension $248$ (\cite{bourbaki}, \S VI.4.10).
In this case the simple root $\alpha$ is, in the labeling of \cite{bourbaki},
$\alpha_8$.  The root subspaces within $\g_0$ are then generated by the
$x_\beta$ where $\beta$ is a root of $\alpha$-height~$0$; that is, a root
of the Lie algebra $E_7$.  There are $126$ such roots
(\cite{bourbaki}, \S VI.4.11); combined with the
$8$-dimensional torus of $E_8$, we have $\dim \g_0 = 134$.  Thus $\go$ is
the subgroup $E_7$ plus a one-dimensional torus, so $\goss$ is $E_7$.

Since $\dim \g_{-2} = \dim \g_2 = 1$, we have $\dim \g_{-1} = \dim \g_1 = 56$.
We see that the action of $\goss$ on $\g_1$ is irreducible since the dense orbit
cannot be contained in any proper subspace, so $\g_1$ is the well-known
minuscule representation of $E_7$.

It has been known since Cartan, in the case where $F=\C$, that there is
a quartic form on the minuscule representation, $V$, of $E_7$
that is invariant under $E_7$ (\cite{cartan1}, p.~274\footnote{It should be noted
that the quartic form is given incorrectly by Cartan; the error seems to have been
first observed by Freudenthal (\cite{freud53}).}).  Freudenthal (\cite{freud53})
later found that the subgroup of $\GL(V)$ stabilizing this quartic form and a
skew-symmetric bilinear form is exactly $E_7$ in this case.  In this section we
use our techniques to establish the subgroup stabilizing the quartic form and the
subgroup stabilizing both forms in our more general context.

\begin{thm}\label{tstabq}
For $G=E_8$, the subgroup of $\GL(\g_1)$ stabilizing the quartic form,
$\Stab(q)$, is generated by $E_7$ and $\mu_4$, where $\mu_4$ is the
group of the fourth roots of unity.
\end{thm}
\begin{proof}
First, $E_7 = \goss$ stabilizes the quartic form by Lemma~\ref{lgoss}.
Also, for $k\in\mu_4$, we have $q(k\cdot x) = k^4 q(x) = q(x)$ for
any $x\in\g_1$, so $\mu_4$ also preserves the quartic form.  Thus
$\Stab(q)$ contains the group generated by $E_7$ and $\mu_4$.

To show the reverse inclusion, suppose $g\in\Stab(q)$.
Let $v = x_\alpha + x_{\rho-\alpha}$.  Since $v$ is in the dense orbit, we have
by Proposition~\ref{porbits} that $q(v)\ne0$ and also, since $q(g\cdot v) = q(v) \ne 0$,
that $g\cdot v$ is in the dense orbit.  Thus there exists some $z\in E_7$
such that $zg\cdot v=kv$ for some $k\in F^\times$.  Let $g'=zg$; then $g'$ is also
in $\Stab(q)$, so $q(v)=q(g'\cdot v) = k^4 q(v)$.  Thus $k\in\mu_4$.  Let
$g''=k^{-1}g'$, then $g''\cdot v=v$, so $g''$ both stabilizes $q$ and fixes $v$.

Lemma~\ref{lstabqv} below, which is the key to the proof, shows that
any element that stabilizes $q$ and fixes $v$ is in the group generated by $E_7$
and $\mu_4$; thus $g''$ is in that group and so is $g$.
\end{proof}

Before completing the proof, we use the preceding theorem to determine
the group that stabilizes both $q$ and the bilinear form $\gen{-,-}$.

\begin{cor}\label{cstabqb}
For $G=E_8$, the subgroup of $\GL(\g_1)$ stabilizing both the quartic
form and the skew-symmetric bilinear form, $\Stab(q,\gen{-,-})$, is $E_7$.
\end{cor}
\begin{proof}
The previous proposition and the fact that $E_7$ stabilizes both forms
yield the following containments:
\[
E_7 \subseteq \Stab(q,\gen{-,-}) \subseteq \Stab(q) = \gen{E_7, \mu_4}.
\]
Let $L_0$ be the root lattice of $E_7$ and $L_1$ its weight lattice.
Then $L_1/L_0$ is a group with two elements
(see, for example, \cite{humlie}, \S 13.1 or \cite{steinberg}, p.~45).  
From \cite{steinberg}, p.~45, the center of $E_7$ is isomorphic to
$\Hom(L_1/L_0,F^\times)$, so the center of $E_7$ consists of the
elements $1$ and $-1$.  Thus
the group $\gen{E_7, \mu_4}$ has two components:  $E_7$ and $i E_7$, where
$i$ is a primitive fourth root of unity.  However, $i$ is not in
$\Stab(q,\gen{-,-})$ since $\gen{ix,iy} = -\gen{x,y}$ for any $x,y\in\g_1$.
Therefore $\Stab(q,\gen{-,-})=E_7$.
\end{proof}

In the remainder of this section we complete the proof of
Theorem~\ref{tstabq} by showing that we can adjust an element
that stabilizes $q$ and fixes $x_\alpha+x_{\rho-\alpha}$ to
produce one that preserves even more structure.  We will use
the same approach in the next section, so we define subspaces
of $\g_1$ and forms on them in a way that is valid when $\g$ is any
Lie algebra of type $D$ or $E$.

Let $A$ and $B$ be the eigenspaces in $\g_1$ described in
Proposition~\ref{peigen} corresponding to the eigenvalues
$+1$ and $-1$, respectively.  Thus $A$ is generated by elements $x_\beta$ where
$\beta$ has $\alpha$-height~$1$ and $\gen{\alpha,\beta}=0$,
whereas $B$ is generated by elements $x_\gamma$ where $\gamma$ has
$\alpha$-height~$1$ and $\gen{\alpha,\gamma}=1$.

We define the cubic forms $f_1$ on $A$ and $f_2$ on $B$ as follows:
\[
f_1(a) = \frac16 q(x_\alpha,a,a,a),\qquad f_2(b)=\frac16 q(x_{\rho-\alpha},b,b,b).
\]

\begin{lem}\label{lstabqvaux}
If $g\in\GL(\g_1)$ is an element that stabilizes the quartic form and fixes
$v=x_\alpha+x_{\rho-\alpha}$, then there is an element $g'$ that
preserves the spaces $A$ and $B$ and stabilizes $\gen{-,-}$ and the
cubic forms defined on $A$ and $B$ such that $g'g^{-1}\in\gen{\goss,\mu_4}$.
\end{lem}
\begin{proof}
Let $g$ be an element that stabilizes $q$ and fixes $v$.  By Lemma~\ref{lsrinvt},
the action of $g$ takes strictly regular elements to strictly regular elements, so
$g\cdot x_\alpha$ and $g\cdot x_{\rho-\alpha}$ are strictly regular.  Since $g$ fixes $v$,
we have $v = g\cdot v = g\cdot x_\alpha + g\cdot x_{\rho-\alpha}$.  However, by
Lemma~\ref{lunique}, the expression of $v$ as a sum of two strictly regular elements is
unique, so $g$ must either fix both $x_\alpha$ and $x_{\rho-\alpha}$ or interchange them.
By \S 12.10 in \cite{garibaldi}, there is an element $z\in i \goss$ that interchanges $x_\alpha$
and $x_{\rho-\alpha}$; of course, such an element also stabilizes $q$.  Hence either
$g$ or $zg$ is an element that stabilizes $q$ and fixes $x_\alpha$ and $x_{\rho-\alpha}$;
call whichever element does so $g'$.  Thus we have $g'g^{-1}\in\gen{\goss,\mu_4}$.

Let $W$ be the subspace of $\g_1$ consisting of elements orthogonal to both
$x_\alpha$ and $x_{\rho-\alpha}$; by Corollary~\ref{corthoinvt}, $W$ is
invariant under $g'$.  All the basis elements $x_\beta$ with $\beta$ of
$\alpha$-height~$1$ except for $x_\alpha$ and $x_{\rho-\alpha}$ are in $W$,
and they form a basis of $W$.  Thus $W$ is the direct sum of the $+1$ and
$-1$ eigenspaces of Proposition~\ref{peigen}, the spaces we have named $A$ and $B$.

Let $A'$ be the subspace of elements $x\in W$ such that
$q(x_{\rho-\alpha},x,y,z)=0$ for all $y,z\in W$, and define a cubic
form on $A'$ by $\frac16 q(x_\alpha,x,x,x)$.  Clearly $g'$ preserves $A'$ and
stabilizes the cubic form.  We claim $A'$ is in fact $A$.

On the one hand, if $x_\beta$ is a basis element of the $+1$ eigenspace, then
we have $\gen{\rho-2\alpha,\beta}=1$.  Since $\gen{\rho,\beta}=1$,
it follows that $\gen{\alpha,\beta} = 0$.  By writing elements
$y,z\in W$ as linear combinations of the basis elements,
$q(x_{\rho-\alpha},x_\beta,y,z)$ expands into a linear combination
of terms of the form $q(x_{\rho-\alpha},x_\beta,x_\gamma,x_\delta)$
with $\gamma,\delta$ such that $\gen{\gamma,\alpha}$ and
$\gen{\delta,\alpha}$ are each either $0$ or $1$.
But then we cannot have $(\rho-\alpha) + \beta + \gamma + \delta = 2\rho$, since
$\gen{(\rho-\alpha) + \beta + \gamma + \delta,\alpha} =
-1 + 0 + \gen{\gamma,\alpha} + \gen{\delta,\alpha}$
is at most $1$, but $\gen{2\rho,\alpha} = 2$.  Hence all the terms
$q(x_{\rho-\alpha},x_\beta,x_\gamma,x_\delta)$ are zero, so
$x_\beta$ is in $A'$.  Thus $A\subseteq A'$.

Conversely, if $x\in W$ is not in $A$, then it has a nonzero component
involving some basis element $x_\beta$ with
$\gen{\beta,\alpha} = 1$.  Thus $\gen{\rho-\alpha,\beta} = 0$, so
$\rho-\alpha$ and $\beta$ are orthogonal roots of $\alpha$-height~$1$.
It follows from Lemma $2.4$ in \cite{rohrle} that any such pair
of roots can be extended to a set of four mutually orthogonal roots, say
$\rho-\alpha, \beta, \gamma, \delta$.  By Lemma~\ref{labcd},
$q(x_{\rho-\alpha},x_\beta,x_\gamma,x_\delta)$ is then nonzero, and thus
$q(x_{\rho-\alpha},x,x_\gamma,x_\delta)$ is also nonzero, since no other
component of $x$ contributes to the value of the form.  Thus $x$ is not
in $A'$.  Therefore $A'\subseteq A$.

Interchanging the roles of $x_\alpha$ and $x_{\rho-\alpha}$, we similarly define
$B'$ to be the subspace of elements $x\in W$ such that $q(x_\alpha,x,y,z)=0$ for all
$y,z\in W$, and define a cubic form on $B'$ by $\frac16 q(x_{\rho-\alpha},x,x,x)$.
As before, $g'$ preserves $B'$ and stabilizes the cubic form, and the same argument,
{\it mutatis mutandis}, shows that $B' = B$.

As in the proof of Corollary~\ref{corthoinvt}, since $g'$ stabilizes
the quartic form, it preserves the bilinear form up to a scalar factor
of $\pm1$.  However, since $g'$ fixes $x_\alpha$ and $x_{\rho-\alpha}$ and
$\gen{x_\alpha, x_{\rho-\alpha}}\ne 0$, the scalar factor is $1$;
thus $g'$ preserves $\gen{-,-}$.
\end{proof}

We now apply the preceding general lemma to the specific case $G = E_8$,
thereby completing the proof of Theorem~\ref{tstabq}.

\begin{lem}\label{lstabqv}
When $G = E_8$, the group that stabilizes the quartic form and fixes the element
$v=x_\alpha+x_{\rho-\alpha}$
is contained in the group generated by $E_7$ and $\mu_4$.
\end{lem}
\begin{proof}
We begin by making some observations about the action of the subgroup $E_6$
of $E_7$ on $\g$.  Since $E_7$ fixes $x_\rho$ and $x_{-\rho}$, as shown
in the proof of Lemma~\ref{lgoss}, $E_6$ certainly does as well.  Similarly,
for any basis element of the Lie algebra $E_6$,
i.e., any $x_\beta$ where $\beta$ is a root orthogonal to both $\rho$ and $\alpha$
or any $h_\gamma$ where $\gamma$ is a simple root other than $\alpha=\alpha_8$
or $\alpha_7$, we have $[x_\beta,x_\alpha]=0$ and $[h_\gamma,x_\alpha]=0$
and likewise $[x_\beta,x_{\rho-\alpha}]=0$ and $[h_\gamma,x_{\rho-\alpha}]=0$.
Thus elements of the group $E_6$ fix $x_\alpha$ and $x_{\rho-\alpha}$.

In the proof of Lemma~\ref{lstabqvaux} it was shown that $A$ could be characterized
in terms of $x_\alpha$, $x_{\rho-\alpha}$, orthogonality and the quartic form;
since all these are preserved by elements of $E_6$, $A$ is invariant under $E_6$.

Since $\g_1$ is $56$-dimensional, it follows from Proposition~\ref{peigen}\
that $A$ is $27$-dimensional.  It is known (\cite{mckaypatera}, p.301) that
the minuscule representation of $E_7$ decomposes into the sum of four
representations of $E_6$, two $1$-dimensional and two $27$-dimensional.
Thus $A$ is a $27$-dimensional minuscule representation of $E_6$.

The cubic form $f_1$ on $A$ is defined in terms of $q$ and $x_\alpha$,
so it is stabilized by $E_6$.  However, by \cite{satokimura}, pp.~25--27, we know that the
$E_6$-invariant polynomials on $A$ are generated by a cubic, at least in
characteristic zero.  By \cite{seshadri}, this holds in general characteristic.
Thus $f_1$ is the unique (up to scalar factor) $E_6$-invariant cubic form on $A$,
provided that it is not zero.

To show that $f_1$ is nonzero, take $\alpha$, $\beta$, $\gamma$, $\delta$ to be
four mutually orthogonal roots of $\alpha$-height~$1$.  As in the proof of
Proposition~\ref{porbits}, for $x = x_\beta + x_\gamma + x_\delta$ we have
$xxx = 6 x_\beta x_\gamma x_\delta$, so $f_1(x) = \frac16 q(x_\alpha,x,x,x) =
q(x_\alpha, x_\beta, x_\gamma, x_\delta)$, which is not zero by Lemma~\ref{labcd}.

By Lemma~\ref{lstabqvaux}, if $g$ is an element
that stabilizes $q$ and fixes $v$, there is a $g'\in g\gen{E_7,\mu_4}$
such that $A$ is invariant under $g'$ and the cubic form $f_1$ is
stabilized by $g'$.  That is, $g'$ is in the stabilizer of the $E_6$-invariant
cubic form on the $27$-dimensional minuscule representation of $E_6$;
by \cite{springerveldkamp}, Theorem 7.3.2, that stabilizer
is $E_6$ itself.

Thus $g'\in E_6$; and therefore $g$ is in $\gen{E_7,\mu_4}$.
\end{proof}

\section{The stabilizer of the quartic form: $G = D_4$}\label{secD4}

In this section, we again consider the group stabilizing the quartic
form and the group stabilizing both the quartic and the bilinear
forms on $\g_1$, this time in the case $G=D_4$.

The diagram that results when $\alpha=\alpha_2$
is removed from the Dynkin diagram of $D_4$ consists of three
unconnected vertices; that is, it represents the Lie algebra which is
the product of three copies of $\mysl_2$.  Thus $\g_0$ is $10$-dimensional,
generated by the three pairs of roots $x_{\alpha_i}, x_{-\alpha_i}$
for $i=1,3,4$ and the four-dimensional Cartan subalgebra of $D_4$;
$\goss$ is thus $\SL_2^3$.  Since $D_4$ has dimension 28, there are 18
other roots; setting aside $\rho$ and $-\rho$, we see that
$\g_1$ and $\g_{-1}$ are eight-dimensional.  Here is a
list of the roots $\beta$ of $\alpha$-height~$1$, sorted according to
the eigenspace decomposition of Proposition~\ref{peigen}:
\[
\begin{tabular}{c|c|c}
$\beta$&$\gen{\rho-2\alpha,\beta}$&$\gen{\alpha,\beta}$\\
\hline
$\rho-\alpha$&$3$&$-1$\\
$\alpha+\alpha_1+\alpha_3,\alpha+\alpha_1+\alpha_4,\alpha+\alpha_3+\alpha_4$&$1$&$0$\\
$\alpha+\alpha_1,\alpha+\alpha_3,\alpha+\alpha_4$&$-1$&$1$\\
$\alpha$&$-3$&$2$
\end{tabular}
\]

As mentioned in the introduction, the quartic form $q$ on the
$8$-dimensional space $\g_1$ is the same as that examined by
Bhargava in \cite{bhargava}.

To establish the stabilizer of the quartic form, we follow a similar
strategy to that employed in the proof of Theorem~\ref{tstabq}:
We define the spaces $A$ and $B$ and cubic forms on them as in the
previous section.
We adjust an element $g\in\GL(\g_1)$ that stabilizes the quartic
form to obtain an element that also fixes $x_\alpha + x_{\rho-\alpha}$,
then apply Lemma~\ref{lstabqvaux} to obtain a
$g'$ that preserves the spaces $A$ and $B$ and stabilizes
the cubic forms on them.  In this case $A$ and $B$ are simple enough
so that we can give the cubic forms explicitly and determine a
suitable subgroup of $\GL(\g_1)$ that contains $g'$.

\begin{thm}\label{tstabqd4}
The stabilizer of the quartic form on $\g_1$ when $G=D_4$ is
$\gen{\SL_2^3,\mu_4}\rtimes S_3$, where $S_3$ is the symmetric group corresponding
to the diagram automorphisms of $D_4$.
\end{thm}

It can be shown that $S_3$ acts trivially on $\mu_4$ here (see Section~$12$ of
\cite{garibaldi}), so we could also write the group as $\gen{\SL_2^3\rtimes S_3,\mu_4}$.

\begin{proof}
Since $\goss=\SL_2^3$ and $\mu_4$ both stabilize the quartic form,
$\gen{\SL_2^3,\mu_4}$ is in $\Stab(q)$.  We will now show that the
diagram automorphisms also stabilize the quartic form.

It will suffice to show that a diagram automorphism fixes $x_\rho$ and $x_{-\rho}$. 
By Corollaire $5.5$ bis in \cite{sga3}, Expos\'e $23$, an outer automorphism
of $\g$ may be taken to act on the Chevalley basis elements $x_{\alpha_i}$
corresponding to the simple roots by permuting the subscripts, and to 
act on the elements $h_i = [x_{\alpha_i},x_{-\alpha_i}]$ by applying the
same permutation to the subscripts; thus the elements $x_{-\alpha_i}$
are also permuted in the same way.  We will write $x_\rho$ in terms of
the $x_{\alpha_i}$, and show that this expression is unaltered by a
permutation of the subscripts $1$, $3$ and $4$; the same argument with
the negatives of the roots will show that $x_{-\rho}$ is fixed as well.

The highest root of $D_4$ is $\rho=\alpha_1+2\alpha_2+\alpha_3+\alpha_4$.
We write this as $\rho = \alpha_2 + \alpha_1 + \alpha_3 + \alpha_4 + \alpha_2$;
in this expression each partial sum is also a root.  Thus we have
\begin{equation}\label{eqc0}
x_\rho = c [x_{\alpha_2}, [x_{\alpha_4}, [x_{\alpha_3}, [x_{\alpha_1}, x_{\alpha_2}]]]],
\end{equation}
where $c$ is a constant (in fact, $c=\pm1$ since all the roots are
long and thus the structure constants are $\pm1$).  Our claim is that this
expression is unaltered when the factors $x_{\alpha_1}$, $x_{\alpha_3}$,
$x_{\alpha_4}$ are permuted.

To verify the claim for the permutation that
interchanges $1$ and $3$, we must show that
$[x_{\alpha_3}, [x_{\alpha_1}, x_{\alpha_2}]] = [x_{\alpha_1}, [x_{\alpha_3}, x_{\alpha_2}]]$;
this is equivalent to the following structure constant equation:
\begin{equation}\label{eqc1}
c_{\alpha_1,\alpha_2} c_{\alpha_3,\alpha_1+\alpha_2} =
 c_{\alpha_3,\alpha_2} c_{\alpha_1,\alpha_2+\alpha_3}.
\end{equation}
To obtain \eqref{eqc1}, we apply \eqref{eqcarter4} with $\beta=\alpha_1+\alpha_2$,
$\gamma=\alpha_2+\alpha_3$, $\delta=-\alpha_2$ and
$\epsilon=-\alpha_1-\alpha_2-\alpha_3$; this yields
\begin{multline*}
c_{\alpha_1+\alpha_2,\alpha_2+\alpha_3} c_{-\alpha_2,-\alpha_1-\alpha_2-\alpha_3} +{}\\
	 c_{\alpha_2+\alpha_3,-\alpha_2} c_{\alpha_1+\alpha_2,-\alpha_1-\alpha_2-\alpha_3} +
	 c_{-\alpha_2,\alpha_1+\alpha_2} c_{\alpha_2+\alpha_3,-\alpha_1-\alpha_2-\alpha_3} = 0.
\end{multline*}
The sum $\alpha_1+2\alpha_2+\alpha_3$ has $\alpha$-height~$2$ but is not equal to $\rho$,
so it is not a root; thus the first term is zero.  Applying the rules for
structure constants from Section~\ref{secforms}, we have
$c_{\alpha_2+\alpha_3,-\alpha_2}=c_{-\alpha_2,-\alpha_3}=c_{\alpha_3,\alpha_2}$,
$c_{\alpha_1+\alpha_2,-\alpha_1-\alpha_2-\alpha_3}=c_{\alpha_3,\alpha_1+\alpha_2}$,
$c_{-\alpha_2,\alpha_1+\alpha_2}=c_{-\alpha_1,-\alpha_2}=-c_{\alpha_1,\alpha_2}$,
$c_{\alpha_2+\alpha_3,-\alpha_1-\alpha_2-\alpha_3}=c_{\alpha_1,\alpha_2+\alpha_3}$.
Thus we have $c_{\alpha_3,\alpha_2}c_{\alpha_3,\alpha_1+\alpha_2}-c_{\alpha_1,\alpha_2}
c_{\alpha_1,\alpha_2+\alpha_3}=0$.  Since all the structure constants involved are
$\pm1$, this is equivalent to the statement that their product is $1$;
this in turn is equivalent to \eqref{eqc1}.

By permuting the roots in the expression for $\rho$, the same argument
applies to any transposition of two of the subscripts $1$, $3$ and $4$.
Since all the transpositions fix $x_\rho$ and $x_{-\rho}$, all the
diagram automorphisms do.  Thus $\gen{\SL_2^3,\mu_4}\rtimes S_3$ is
contained in the stabilizer of the quartic form.

We now consider the reverse inclusion.
Let $v = x_\alpha + x_{\rho-\alpha}$.  As in the proof of
Theorem~\ref{tstabq}, given some $g\in\GL(\g_1)$ which
stabilizes $q$, there exists some $z\in\goss$ such that
$zg\cdot v$ is a scalar multiple of $v$, and there is some
$k\in\mu_4$ such that $g''=k z g$ fixes $v$ and still
stabilizes $q$.

Applying Lemma~\ref{lstabqvaux} to $g''$, we obtain an element
$g'$ that preserves $A$ and $B$ and stabilizes $\gen{-,-}$ and
the cubic forms on $A$ and $B$.

By definition, the subspace $A$ is generated by the root subspaces corresponding to
roots orthogonal to $\alpha$; examining the list of roots in $\g_1$,
these are $\beta=\alpha+\alpha_1+\alpha_3$,
$\gamma=\alpha+\alpha_1+\alpha_4$ and $\delta=\alpha+\alpha_3+\alpha_4$.
We easily check that $\alpha$, $\beta$, $\gamma$ and $\delta$ are mutually
orthogonal.  For an arbitrary element
$x = \lambda_1 x_\beta + \lambda_2 x_\gamma + \lambda_3 x_\delta$ of $A$,
we find that the cubic form is
\[
\frac16 q(x_\alpha,x,x,x)=\lambda_1\lambda_2\lambda_3 q(x_\alpha,x_\beta,x_\gamma,x_\delta),
\]
since the terms with a repeated argument are zero by Lemma~\ref{laab}.
By Proposition~\ref{pq}, this is $\epsilon \lambda_1\lambda_2\lambda_3$,
where $\epsilon=\pm1$ is a product of structure constants.

Let $T=(a_{ij})$, $1\le i,j\le 3$, be the matrix of the linear transformation
on~$A$ given by $x\mapsto g'\cdot x$ with respect to the basis
$x_\beta$, $x_\gamma$, $x_\delta$.  The value of the cubic form is the same for
$x = \lambda_1 x_\beta + \lambda_2 x_\gamma + \lambda_3 x_\delta$ and
$g'\cdot x$, so we have
\[
\lambda_1\lambda_2\lambda_3 = (a_{11}\lambda_1+a_{12}\lambda_2+a_{13}\lambda_3)
    (a_{21}\lambda_1+a_{22}\lambda_2+a_{23}\lambda_3)
    (a_{31}\lambda_1+a_{32}\lambda_2+a_{33}\lambda_3)
\]
for all $\lambda_1,\lambda_2,\lambda_3\in F$.
By unique factorization in $F[\lambda_1,\lambda_2,\lambda_3]$,
the three factors on the right-hand side are (up to units)
$\lambda_1,\lambda_2,\lambda_3$, say $c_1\lambda_1,c_2\lambda_2,c_3\lambda_3$,
with $c_1 c_2 c_3 = 1$.  If the factors occur in that order, then
$T$ is diagonal, with the third entry determined by the first two; each
such $T$ corresponds to an element $(c_1,c_2,c_3)$ of $\Gm\times \Gm\times \Gm$
for which the product of the three components is 1.  However, the order of the
factors may be different, so in general $T$ may be an element
of $(\Gm\times \Gm\times \Gm)\rtimes S_3$.

The subspace $B$ is generated by the root subspaces corresponding to
the roots $\rho-\beta=\alpha+\alpha_4$, $\rho-\gamma=\alpha+\alpha_3$
and $\rho-\delta=\alpha+\alpha_1$.  As $\alpha,\beta,\gamma,\delta$
are mutually orthogonal, so are $\rho-\alpha,\rho-\beta,\rho-\gamma,\rho-\delta$.
The cubic form on $B$ is given by $\frac16 q(x_{\rho-\alpha},x,x,x)$;
for $x= \lambda_1 x_{\rho-\beta} + \lambda_2 x_{\rho-\gamma} + \lambda_3 x_{\rho-\delta}$
this is, as in the previous case, $\pm \lambda_1\lambda_2\lambda_3$.
As before, $g'$ must map $x_{\rho-\beta}$, $x_{\rho-\gamma}$ and
$x_{\rho-\delta}$ to scalar multiples of the same basis elements, possibly
permuted.

However, since $g'$ stabilizes $\gen{-,-}$, the action of $g'$ on $B$
can be computed given its action on $A$.  
Suppose, for example, that $g'$ maps $x_\beta$ to $c x_\gamma$ in $A$, then
$\gen{x_\beta,x_{\rho-\beta}}=\gen{c x_\gamma, g'\cdot x_{\rho-\beta}}$;
since this must be $c_{\beta,\rho-\beta}$, we have that $g'\cdot x_{\rho-\beta}$
is necessarily $c_{\beta,\rho-\beta} c_{\gamma,\rho-\gamma} c^{-1} x_{\rho-\gamma}$.
In general, $\beta$ and $\gamma$ may be replaced by any of $\beta$,
$\gamma$ or $\delta$, with a similar result.
Hence the action of $g'$ on $B$ is determined by its action on $A$;
in particular, if acts diagonally on $A$, it also does so on $B$.

It remains only to show that an element $g'$ that corresponds to
element of $\Gm\times \Gm\times \Gm$ is an element of $\SL_2^3$.
We will consider the action of an element of $\SL_2^3$ that corresponds
to an element of $\h$ of the form $t_1 h_{\alpha_1} + t_3 h_{\alpha_3} + t_4 h_{\alpha_4}$.
By Lemma~$19$(c) in \cite{steinberg}, the action of the element corresponding
to $t_1 h_{\alpha_1}$ takes $x_\beta$ to $t_1^{\gen{\beta,\alpha_1}} x_\beta$,
which is $t_1 x_\beta$ since $\gen{\beta,\alpha_1} = 1$.  Similarly, it
takes $x_\gamma$ to $t_1 x_\gamma$ since $\gen{\gamma,\alpha_1} = 1$ and
takes $x_\delta$ to $t_1^{-1} x_\delta$ since $\gen{\delta,\alpha_1} = -1$;
thus its action on $A$ is that of the element $(t_1,t_1,t_1^{-1})$ in
$\Gm\times \Gm\times \Gm$.  In the same fashion,
we find that $t_3 h_{\alpha_3}$ corresponds to $(t_3,t_3^{-1},t_3)$ and
$t_4 h_{\alpha_4}$ to $(t_4^{-1},t_4,t_4)$.  Since these classes of
elements are multiplicatively independent, they generate
$\Gm\times \Gm\times \Gm$; the elements with the product of the
components equal to $1$ come from elements of the form
$t_1 h_{\alpha_1} + t_3 h_{\alpha_3} + t_4 h_{\alpha_4}$ with
$t_1 t_3 t_4 = 1$.  Since $\gen{\alpha,\alpha_i} = -1$ for $i=1,3,4$,
this element takes $x_\alpha$ to $t_1^{-1} t_3^{-1} t_4^{-1} x_\alpha = x_\alpha$,
so it fixes $x_\alpha$ just as $g'$ does.   The action on the remaining basis
elements, namely $x_{\rho-\alpha}$ and those of $B$, must also correspond to
that of $g'$ because an element of $\SL_2^3$
stabilizes the bilinear form.

Thus $g'$ is in $\SL_2^3\rtimes S_3$, from which it follows that
the original $g\in\GL(\g_1)$ stabilizing the quartic form is in
$\gen{\SL_2^3,\mu_4}\rtimes S_3$.
\end{proof}

The determination of the group that stabilizes both $q$ and the
bilinear form $\gen{-,-}$ is parallel to Corollary~\ref{cstabqb}.

\begin{cor}\label{cstabqbd4}
In the case $G=D_4$, the subgroup of $\GL(\g_1)$ stabilizing both the quartic
form and the skew-symmetric bilinear form, $\Stab(q,\gen{-,-})$, is
$\SL_2^3\rtimes S_3$
\end{cor}
\begin{proof}
The previous theorem and the fact that $\SL_2^3$ and the
diagram automorphism stabilize both forms
yield the following containments:
\[
\SL_2^3\rtimes S_3 \subseteq \Stab(q,\gen{-,-}) \subseteq \Stab(q) = \gen{\SL_2^3,\mu_4}\rtimes S_3.
\]
Since $-1\in\SL_2$, we also have $-1\in\SL_2^3$.  Thus
$\SL_2^3\rtimes S_3$ is an index $2$ subgroup of $\gen{\SL_2^3,\mu_4}\rtimes S_3$.
However, the coset containing $i$, a primitive fourth root of unity, is not in
$\Stab(q,\gen{-,-})$ since $\gen{ix,iy} = -\gen{x,y}$ for any $x,y\in\g_1$.
Therefore $\Stab(q,\gen{-,-})=\SL_2^3\rtimes S_3$.
\end{proof}

\section{Non-split groups}\label{secnonsplit}

In the preceding sections, we assumed that $G$ was split over $F$.  This was only for convenience;
in this section we will show that most of our results hold for quite general Freudenthal triple systems.

Suppose $G$ is an absolutely almost simple linear algebraic group, not of type $A$ or $C$, over
a field $F$ of characteristic $\ne 2,3$.  Fix a maximal $F$-torus $T$, which we may assume contains
a maximal $F$-split torus, and also fix a set $\Delta$ of simple roots for $G$ with respect to $T$
over a separable closure $\Fs$ of $F$.

There is a uniquely determined root $\alpha\in\Delta$ as in Section~\ref{secprelim}.  We require
that, in the Tits index of $G$ as defined in \cite{tits}, the vertex $\alpha$ is circled.  This
is equivalent to having an $F$-homomorphism $\rho^\vee \colon \Gm \rightarrow T$
corresponding to the coroot $\rho$ (i.e., such that $\Lie(\im \rho^\vee) \otimes \Fs$ is $\Fs h_\rho$);
see Corollaire $6.9$ in \cite{boreltits1}.  We grade the Lie algebra $\g$ of $G$ by setting 
\[
\g_i := \{x\in\g \mid \text{$\rho^\vee(t) x = t^i x$ for all $t\in\Fs^\times$}\}.
\]
When $G$ is split (e.g., if we extend scalars to $\Fs$), we obtain the same grading as in
Section~\ref{secprelim}.  We choose a nonzero vector $x_\rho\in\g_2$, which gives a
skew-symmetric bilinear form $\gen{-,-}$ and a quartic form~$q$ on $\g_1$ by the same formulas
as in Section~\ref{secprelim}.

Now Lemmas/Propositions/Theorems/Corollaries \ref{lgoss}, \ref{lnondeg}, \ref{lxxy}, \ref{porbit1},
\ref{prank1}, \ref{leq5}, \ref{tFTS}, \ref{lsrinvt}, \ref{pgroups} and \ref{corthoinvt} all hold
over~$F$ without any change in their statements.  Indeed, it suffices to verify each
over~$\Fs$, where $G$ is split.

Theorems/Corollaries \ref{tstabq}, \ref{cstabqb}, \ref{tstabqd4} and \ref{cstabqbd4}
can be viewed as determining the $\Fs$-isomorphism class
of their respective stabilizer groups (which are defined over $F$).

For readers interested in Freudenthal triple systems, we now suppose that we are given such a triple
system $(V, q, \gen{-,-})$---denoted briefly by $V$---defined over $F$ such that
$V\otimes\Fs$ can be identified with one of the triple systems constructed in
Sections~\ref{secprelim}--\ref{secfts}.  We claim $V$ can be constructed from some group $G$
defined over $F$ by using the construction given earlier in this section and thus the results listed
also hold for $V$.

We illustrate the claim in the case where $H = \Stab(q, \gen{-,-})$ is of type $E_7$;
equivalently, $V\otimes\Fs$ is obtained from a group of type $E_8$.  Since the $56$-dimensional
representation of $H$ is defined over $F$, $H$ is obtained by twisting the split simply-connected
group $E_7^\sc$ of type $E_7$ by a $1$-cocycle $\eta$ (in Galois cohomology) with values in
$E_7^\sc(\Fs)$.  If the split group of type $E_8$ (which naturally contains $E_7^\sc$) is also
twisted by $\eta$, we find a copy of $H$ inside a group $G$ of type $E_8$.  The construction above now
yields a Freudenthal triple system $V'$ with automorphism group $H$, which must be similar to
$V$ by \cite{garibaldi2}, Theorem 4.16(2).  By scaling $x_\rho$, we can arrange for $V'$ to be isomorphic
to $V$, as desired.

In addition to the $56$-dimensional representation of a group of type $E_7$,
we see in the same way that the results of this paper apply to the
half-spin representation of a group of type $D_5$, the natural
$20$-dimensional representation of a group of type $A_5$ and the natural
$8$-dimensional representation of a group of type $A_1 \times A_1 \times A_1$,
whenever such representations are defined over $F$.

\section*{Acknowledgment}

{\small
This paper is based on my Ph.D. dissertation, \cite{helenius}, at Emory University,
and could not have been created without the assistance and guidance of my advisor,
Skip Garibaldi.
}

%\bibliography{FTS}

\end{document}